\documentclass[11pt,a4paper,reqno]{amsart} 
\usepackage{amsmath}
\usepackage{subcaption}
\usepackage{graphics}
\usepackage{epsfig}
\usepackage{graphicx}  
\usepackage{epstopdf}
\usepackage[pagewise]{lineno}
\usepackage[colorlinks=true]{hyperref}
\newtheorem{lemma}{Lemma}[section]

\newtheorem{thm}[lemma]{Theorem}
 
\theoremstyle{definition} 
\newtheorem{Def}[lemma]{Definition}
\newtheorem{exam}[lemma]{Example}  \theoremstyle{remark} 
\newtheorem{rem}[lemma]{Remark}

 \newcommand{\B}{\mathcal{B}}

\title{{\footnotesize Product of Expansive Markov Maps with Hole}} 
\author[Haritha C]{Haritha C}
\address{Department of Mathematics\\
Indian Institute of Science Education and Research Bhopal\\
Bhopal Bypass Road, Bhauri \\
Bhopal 462 066, Madhya Pradesh\\
India}
\email{charitha@iiserb.ac.in}
\author[Nikita Agarwal]{Nikita Agarwal}
\address{Department of Mathematics\\
Indian Institute of Science Education and Research Bhopal\\
Bhopal Bypass Road, Bhauri \\
Bhopal 462 066, Madhya Pradesh\\
India}
\email{nagarwal@iiserb.ac.in}

\date{\today}

\begin{document}

\begin{abstract}
	We consider product of expansive Markov maps on an interval with hole which is conjugate to a subshift of finite type. For certain class of maps, it is known that the escape rate into a given hole does not just depend on its size but also on its position in the state space. We illustrate this phenomenon for maps considered here. We compare the escape rate into a connected hole and a hole which is a union of holes with a certain property, but have same measure. This gives rise to some interesting combinatorial problems. 
\end{abstract}

\maketitle 

%\tableofcontents
\section{Introduction}
Open dynamical systems or dynamical systems with a hole were first proposed by Pianigiani and Yorke in~\cite{PY}. Such systems are interesting because of their dynamical properties and also their applications, we refer to~\cite{BC, BKT, BDM, CM, Dem, DemYoung, DWY} for some related work. In~\cite{BY}, and also~\cite{AB, KL}, the question of whether escape rate into two holes of same measure is the same, was asked. They gave an affirmative answer when there exists a group of measure preserving translations of the state space which commute with the dynamics. For a system with strongly chaotic dynamics, they proved that the escape will be faster through a hole where the minimal period (minimum among periods of all periodic points in that hole) is larger. They considered various classes
of dynamical systems with strongly chaotic behaviour and Markov holes. In particular, if the dynamical system is conjugate to a full shift, symbolic dynamics was used to obtain these results. 

In this paper, we discuss a similar question as the one asked by~\cite{BY}. Here we consider product of expanding Markov maps. Such maps are conjugate to a subshift of finite type where the transition matrix and Markov partition are induced by the individual maps in the product. We consider Markov rectangles which are product of Markov intervals of individual factors, and compare the escape rates into two such rectangles with the same measure. For the product map, a Markov rectangle corresponds to a union of cylinders, whereas each Markov interval in the product corresponds to a single cylinder, in the associated shift spaces. In~\cite{BY}, this was the key idea to obtain results. 

We also compare the escape rates into two holes with the same measure, where one hole is a Markov rectangle which corresonds to a single cylinder and the other is a union of several (basic) Markov rectangles of smaller measure with a certain property. We compute the escape rate using two techniques, one is using the $M$-step shift, as it is popularly known. The other is using the combinatorial technique, as given in~\cite{Combinatorial}. The first technique reduces to computing the corresponding transition matrix and its largest eigenvalue in modulus, which is real and positive by the Perron-Frobenius theorem (if the matrix is irreducible). The second technique involves solving the generating function for the number of allowed sequences which never hit the hole. Here, the problem reduces to solving a recurrence relation. Depending on the situation, we will choose one of these techniques to compute the escape rate. In certain cases, we will employ both techniques for illustration purposes. The combinatorial technique is accessible when the number of forbidden words is one (at most two), or when the correlation matrix has a `nice' form, which we have exploited in Section~\ref{sec:compare}. In all other cases, the other approach is useful. But it should be noted that, if the size of the transition matrix is large, it is difficult to obtain general results.

In Section~\ref{sec:prelim}, we give preliminaries and the set-up of the broad class of maps that we have considered in this paper. In Section~\ref{sec:prod}, we consider product of expansive Markov maps and discuss techniques to compute the escape rate into a hole. In Section~\ref{sec:torus_map}, we consider a product of two $k$-expanding transformations and holes which are products of Markov intervals of individuals factors. A Markov rectangle is a union of cylinders of a certain type, under the conjugacy with the full shift space. In Section~\ref{sec:compare}, we compare the escape rate into a basic rectangle with the escape rate into a union of basic rectangles of identical measure (with certain property), having same total measure as the single basic rectangle, see Theorems~\ref{thm:large} and~\ref{thm:large1}. This collection of rectangles give rise to interesting combinatorial problems (independent of their dynamical interest) which are discussed towards the end of Section~\ref{sec:conc}. In Section~\ref{sec:subshift}, we discuss maps which are conjugate to a subshift of finite type (and not necessarily with the full shift space). We will observe that the escape rate does not necessarily depend on the minimal period in the hole. Finally we give concluding remarks in Section~\ref{sec:conc}. 

\section{Preliminaries}\label{sec:prelim}
\subsection{Poincar\'e Recurrence Time and Escape Rate}
Let $(X,\B,\mu)$ be a probability space and $T:X\to X$ be a measure-preserving transformation (that is, $T$ is a measurable map and $\mu(T^{-1}A)=\mu(A)$ for every $A\in\B$).
Consider $A\in\B$, any measurable set with $\mu(A)>0$.

\begin{Def}
	The \textit{Poincar\'e recurrence time} of $A$ under $T$ is defined as the positive integer $\tau(A)$ given by 
	\[
	\tau(A)=\inf_{n \geq 1}\{n\ \vert\ \mu(T^{-n}(A)\cap A)>0\}.
	\]
\end{Def}

\noindent If $\tau(A)$ is finite, then it is the minimum $n\geq 1$ at which the points in $A$ whose orbit under $T$ intersects with $A$ has positive measure. By Poincar\'e recurrence theorem, $\tau(A)$ is always finite. 

For a given $A\in\B$ with $\mu(A)>0$, known as ``hole", consider the restriction map $T\vert_{(X\setminus A)}:X\setminus A\to X$. The iterates of this map given by $T\vert_{(X\setminus A)}^n$ are well-defined as long as the orbit of $x\in X\setminus A$ under the map $T\vert_{(X\setminus A)}$ does not escape into the hole $A$ (that is, $T^n(x)\notin A$ for all $n$). 

\begin{Def}\label{escape-rate}
	The \textit{escape rate} into a hole $A\in\B$ is defined as the non-negative number given by 
	\[
	\rho(A)=-\lim_{n \to \infty}\frac{1}{n}\ln \mu(X\setminus \Omega_n(A)),
	\]
	provided the limit exists, where for $n \geq 0$,
	\[
	\Omega_n(A)=\{x \in X \ \vert \ T^j(x)\in A \text{ for some } 0\leq j\leq n\}=\bigcup_{j=0}^n T^{-j}(A).
	\] 
\end{Def}

\noindent Escape rate represents the average rate at which the orbits escape into the hole. Larger the escape rate, faster the orbits terminate.

\begin{Def}
	Let $T_1$ and $T_2$ be two measure-preserving transformations on the probability spaces $(X_1,\B_1,\mu_1)$ and $(X_2,\B_2,\mu_2)$, respectively. The transformations $T_1$ and $T_2$ are said to be \textit{metrically conjugate} if there exists $A_i\in\B_i$ with $\mu_i(A_i)=1$ and $T_i(A_i)\subseteq A_i$ for $i=1,2$, and there is an invertible measure-preserving transformation $\psi:A_2\to A_1$ such that $\psi\circ T_2(x)=T_1\circ \psi(x)$, for every $x\in A_2$.
\end{Def}
The map $\psi$ is called the \textit{conjugacy map}. It is shown in ~\cite[Lemma 2.3.5]{BY} that the escape rate is invariant under this conjugacy, that is, if $T_1,T_2$ are metrically conjugate under the conjugacy map $\psi$, then $\rho_{T_2}(A)=\rho_{T_1}(\psi(A))$ for any hole $A\in \B_2$.

\subsection{Markov Maps} 
We refer to~\cite{Markov} for more details.

\begin{Def}
	Let $I=[a,b]$ be a closed interval in $\mathbb{R}$. A collection of closed intervals $\{I_0,I_1,\dots, I_{N-1}\}$ is called a \textit{partition} of $I$ if\begin{enumerate}
		\item[a)] $I=\bigcup_{i=0}^{N-1} I_i$,
		\item[b)] $\text{int}(I_i)\cap \text{int}(I_j)=\emptyset$, for $i\ne j$.
	\end{enumerate}
\end{Def} 

\begin{Def}
	A map $f:I\rightarrow I$ is called a \emph{Markov map} if there exists a partition $\{I_0,I_1,\dots, I_{N-1}\}$ of $I$, called a \emph{Markov partition} such that for all $0\leq i,j\leq N-1$,
	\begin{enumerate}
		\item[a)] either $f(\text{int}(I_i))\cap \text{int}(I_j)=\emptyset$, or
		\item[b)] $\text{int}(I_j)\subseteq  f(\text{int}(I_i))$.
	\end{enumerate}
\end{Def}

\begin{Def}
	The $N\times N$ matrix $A$ defined by
	\[
	A_{ij} = 1, \ \text{if}\ \text{int}(I_j)\subseteq  f(\text{int}(I_i)),\ \text{and}\ A_{ij}=0, \text{ if } f(\text{int}(I_i))\cap \text{int}(I_j)=\emptyset,
	\]
	is called the \emph{topological transition matrix} of the Markov map $f$. (Note: $A_{ij}$=1 if and only if the transition from $I_i$ to $I_j$ is allowed.)
\end{Def}

\begin{Def}
	A Markov map $f$ is said to be \emph{expansive (expanding) Markov map} if $f$ is smooth on each $\text{int}(I_i)$ and there exists $\lambda>1$ such that $\vert f'(x)\vert \geq \lambda$, for all $x\in \text{int}(I_i)$, $i=0,1,\dots,N-1$.
\end{Def}

\begin{Def}
	A sequence $x_0x_1\dots \in \{0,1,\dots,N-1\}^\mathbb{N}$ is said to be \emph{admissible} if $A_{x_nx_{n+1}}=1$, for all $n\geq 0$. Let $\Sigma_A$ denote the collection of all admissible sequences in $\{0,1,\dots,N-1\}^\mathbb{N}$. Let $\sigma:\Sigma_A\rightarrow \Sigma_A$ be the left shift map.
	
	When $A$ has all the entries to be 1, then any sequence in $\{0,1,\dots,N-1\}^{\mathbb{N}}$ is admissible. In this case, $\Sigma_A=\{0,1,\dots,N-1\}^{\mathbb{N}}$ is called the \textit{full shift space}. Otherwise it is called a \textit{subshift of finite type}.
\end{Def}
\begin{exam}
	\begin{figure}
		\centering
		\includegraphics[width=.8\textwidth]{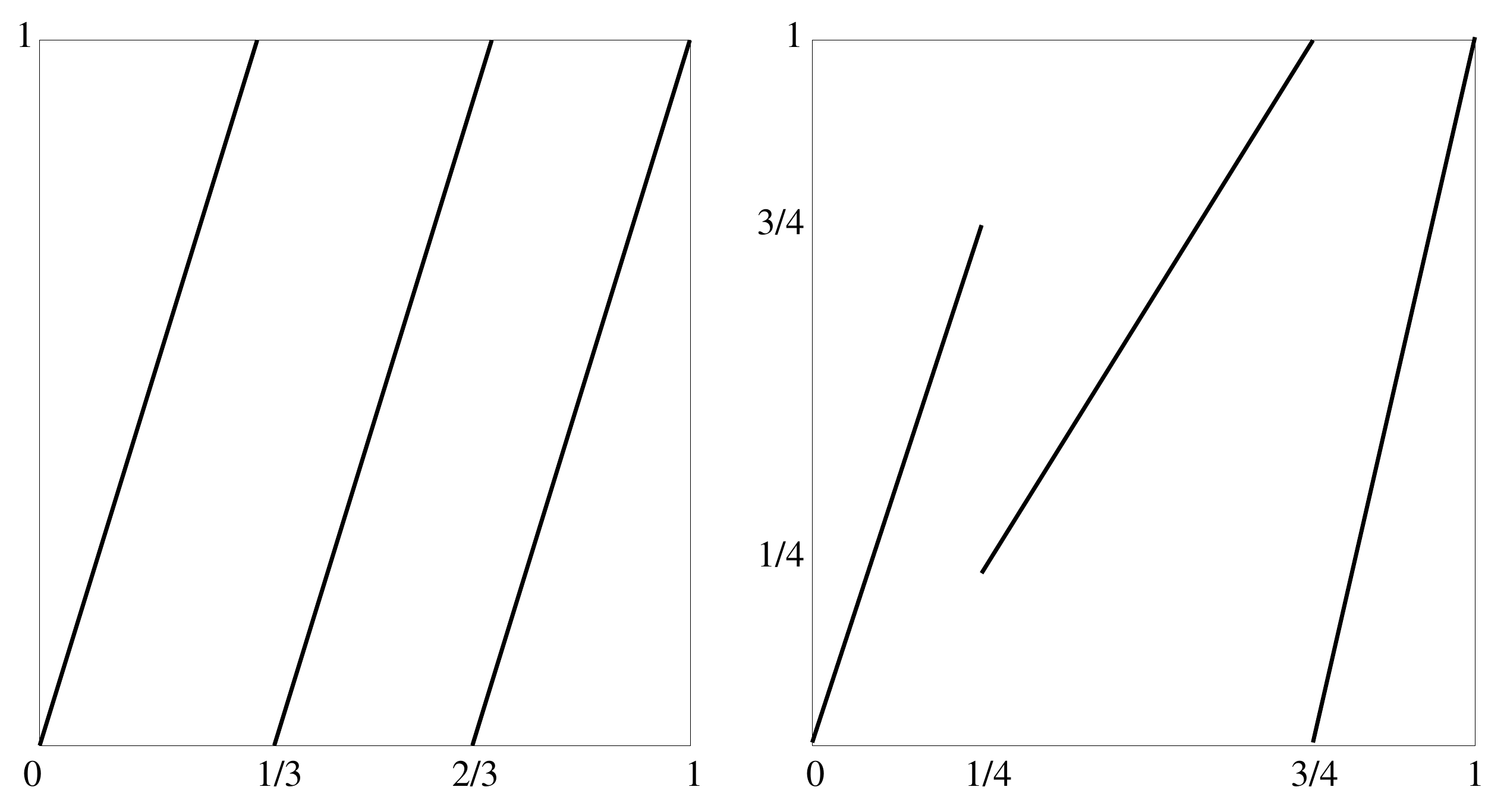}
		\caption{Examples of expansive Markov maps.}
		\label{fig:0}
	\end{figure}	
	Figure~\ref{fig:0} shows two examples of expansive Markov maps. The partition $\{I_0=\left[0,1/3\right], I_1=\left[1/3,2/3\right], I_2=\left[2/3,1\right]\}$ is a Markov partition of $I=[0,1]$ for the map on the left in the figure. The partition $\{I_0=\left[0,1/4\right], I_1=\left[1/4,3/4\right], I_2=\left[3/4,1\right]\}$ is a Markov partition of $I=[0,1]$ for the map on the right in the figure. The corresponding topological transition matrices for the left and the right maps are given by
	\[
	\begin{pmatrix}
	1&1&1\\
	1&1&1\\
	1&1&1
	\end{pmatrix} \text{ and}\ 
	\begin{pmatrix}
	1&1&0\\
	0&1&1\\
	1&1&1
	\end{pmatrix},
	\]
	respectively.  
\end{exam}

\noindent Note that an expansive Markov map $f$ need not be measure preserving with respect to the Lebesgue measure (see the map on the right in Figure~\ref{fig:0}). The following theorem gives the existence of an invariant measure for $f$ and a conjugacy with the left shift map on the shift space $\Sigma_A$.	We refer~\cite[Proposition 4.8]{Pollicott_book} for the following theorem.
\begin{thm}\label{thm:symbonto}
	Let $f:I\rightarrow I$ be an expansive Markov map with transition matrix $A$ and Markov partition $\{I_0,I_1,\dots,I_{N-1}\}$. For any admissible sequence $x_0x_1\dots \in\Sigma_A$ there exists a unique point $x\in I$ such that $f^{n}(x)\in I_{x_n}$, for all $n\geq 0$. This mapping (say $\pi$) from $\Sigma_A$ to the interval $I$ taking $x_0x_1\dots$ to $x$ is onto. Moreover, $\pi\circ \sigma = f\circ \pi$.
\end{thm}
\begin{Def}
	Let $w$ be a finite word that appears as a subword in some sequence in $\Sigma_A$.
	The \textit{cylinder} based at  $w$ is defined as a subset of $\Sigma_A$ consisting of all the one-sided sequences in $\Sigma_A$ that start with $w$. It is denoted as $C_w$. 
\end{Def}
\noindent To illustrate the above definition, $C_{01}=\{01x_1x_2\dots\ \vert\   01x_1x_2\dots\in \Sigma_A\}$. 
\begin{rem} \label{rem:meas}
	1) In Theorem~\ref{thm:symbonto}, it is enough to assume that an iterate of $f$ is expansive (rather than $f$ itself).\\
	2) The map $\pi$ is one-one except on a countable set.\\
	3) If $A$ is irreducible (that is, for each pair of indices $0\leq i,j\leq N-1$, there exists $k\geq 1$ such that $(A^k)_{ij}>0$), we define a Markov measure (known as \textit{Parry measure}) $\mu$ on $\Sigma_A$ as follows;
	Let $w=i_1i_2\dots i_n$ be a word of length $n\geq 2$ that appears as a subword in some sequences in $\Sigma_A$, and let
	$C_w$ be the cylinder based at $w$. We define
	\[
	\mu(C_w)=\frac{u_{i_1}v_{i_n}}{\lambda^{n-1}}a_{i_1i_2}a_{i_2i_3}\dots a_{i_{n-1}i_n},
	\]
	where $\lambda$ is the largest eigenvalue of $A=(a_{ij})$ (which is both real and positive by Perron-Frobenius theorem) and $v = (v_0,v_1,\dots,v_{N-1})^T$, $u = (u_0,u_1,\dots,u_{N-1})$
	are the normalized right and left eigenvectors of $A$ with respect to the eigenvalue $\lambda$ such that $uv=1$. Note that $\mu(C_w)=u_iv_i$ if $w=i$ for $0\leq i\leq N-1$. This definition can be extended to the $\sigma$-algebra generated by cylinders based at all the allowed words of finite length. Note that the shift map $\sigma$ on $\Sigma_A$ is measure preserving with respect to this Markov measure.\\
	For the full shift in $N$ symbols, since all the entries of $A$ are $1$, $\mu(C_w)=\frac{1}{N^n}$ for all the words $w$ of length $n$. Hence the cylinders based at words of same length have same measure, which is not true for a general shift space.\\
	4) Since the map $\pi$ is a bijection except on a measure zero set in $\Sigma_A$, we define a measure $\tilde{\mu}$ on $I$ induced from the measure $\mu$ on $\Sigma_A$. We say that  $B\subseteq I$ is $\tilde{\mu}$-measurable set if $\pi^{-1}(B)\subseteq \Sigma_A$ is $\mu$-measurable and the measure $\tilde{\mu}$ on $I$ is given as $\tilde{\mu}(B)=\mu(\pi^{-1}(B))$. With respect to this new measure $\tilde{\mu}$, the map $f$ is measure-preserving and $\pi$ is a conjugacy between $f$ and $\sigma$. 
\end{rem}

\subsection{Product of expansive Markov maps}\label{sec:prod_emm} 
Let $f_1,\dots,f_k:I\rightarrow I$ be expansive Markov maps. Let $\{I_0^i,I_1^i,\dots, I_{N_i-1}^i\}$ be a Markov partition of $f_i$ and $A_i$ be the associated $N_i\times N_i$ transition matrix, for $i=1,\dots,k$. Consider the product map $f=f_1\times \dots\times f_k:I^k\rightarrow I^k$ defined as 
\begin{eqnarray}\label{eq:product}
	f(x_1,\dots,x_k) = (f_1(x_1),\dots,f_k(x_k)),
\end{eqnarray}
for all $x_1,\dots,x_k\in I$.

\subsubsection{Markov partition} Set $N=N_1\dots N_k$. The $N$ rectangles $\{R_{i_1,\dots,i_k}=I_{i_1}^1\times\dots\times I_{i_k}^k\ \vert \ 0\leq i_j\leq N_j-1,\ 1\leq j\leq k\}$ form a Markov partition of $I^k$ in the sense that for all $0\leq i_j,i_j'\leq N_j-1$, $1\leq j\leq k$,
\begin{enumerate}
	\item[a)] either $f(\text{int}(R_{i_1,\dots,i_k}))\cap \text{int}(R_{i_1',\dots,i_k'})=\emptyset$, or\item[b)] $\text{int}(R_{i_1',\dots,i_k'})\subseteq  f(\text{int}(R_{i_1,\dots,i_k}))$.
\end{enumerate}

\noindent Define a bijection $\phi$ from $\{0,1,\dots, N-1\}$ to $\{(i_1,\dots,i_k) \ \vert \ 0\leq i_j\leq N_j-1,\ 1\leq j\leq k\}$ as follows:\\
For each $0\leq n\leq N-1$, we can write $n=\sum_{j=2}^k i_{j-1} N_kN_{k-1}\dots N_{j}+i_k$ with $0\leq i_j\leq N_j-1$. Set 
\[
\phi(n) = (i_1,\dots,i_k).
\]
The map $\phi$ is well-defined. Using $\phi$, we can index the rectangles as $\{R_n\sim R_{\phi(n)}\ \vert\ 0\leq n\leq N-1\}$.

\subsubsection{Transition matrix} The associated $N\times N$ transition matrix $A$ is given as follows: $A_{nm}=1$ if and only if $(A_j)_{i_ji_j'}=1$ for all $1\leq j\leq k$, where $\phi(n)=(i_1,\dots,i_k)$ and $\phi(m)=(i_1',\dots,i_k')$. In other words, 
\[
A_{nm}=(A_1)_{i_1i_1'}(A_2)_{i_2i_2'}\dots(A_k)_{i_ki_k'}.
\]
Let $\Sigma_A \subseteq \{0,1,\dots,N-1\}^\mathbb{N}$ denotes the collection of all admissible sequences.

\begin{thm}\label{product-factor}
	Let $f:I^k\rightarrow I^k$ be as in~(\ref{eq:product}) with transition matrix $A$. For any admissible sequence $x_0x_1\dots \in \Sigma_A$ there exists a unique point $x\in I^k$ such that $f^{n}(x)\in R_{x_n}$, for all $n\geq 0$. This mapping (say $\pi$) from $\Sigma_A$ to $I^k$ is onto. Moreover, $\pi\circ \sigma = f\circ \pi$.
\end{thm}
\begin{proof}
	Let $\phi(x_n)=(i_n^1,\dots,i_n^k)$, $n\geq 0$. Then for $1\leq j\leq k$, $i_0^ji_1^j\dots$ is an admissible sequence corresponding to the transition matrix $A_j$ for the map $f_j$. Thus by Theorem~\ref{thm:symbonto}, there exists a unique point $x^j\in I$ such that $f_j^n(x^j)\in I_{i_n^j}$, for all $n\geq 0$, $1\leq j\leq k$. Set $x=(x^1,\dots,x^k)\in I^k$. Then $f^n(x)\in I_{i_n^1}\times \dots\times I_{i_n^k}=R_{i_n^1,\dots,i_n^k}=R_{\phi(x_n)}\sim R_{x_n}$.\\
	Consider $f\circ\pi(x_0x_1\dots)=f(x)$ (notation as in the previous part of the theorem), where $f^n(x)\in R_{x_n}$, for $n\geq 0$. Thus $f^n(f(x))\in R_{x_{n+1}}$, for $n\geq 0$. Also $\pi\circ \sigma(x_0x_1\dots)=\pi(x_1x_2\dots)=y$ if and only if $f^n(y)\in R_{x_{n+1}}$, for $n\geq 0$. By uniqueness, we get $f(x)=y$. Hence $\pi\circ \sigma = f\circ \pi$.
\end{proof}

\begin{rem} 
	1) For each $1\leq j\leq k$, $A_j$ has all the entries as 1 if and only if $A$ has all its entries 1. In this case, $\Sigma_A=\{0,1,\dots,N-1\}^\mathbb{N}$. \\
	2) The rows and columns of the matrix $A$ are labelled as $0,1,\dots,N-1$, whose $nm^{\text{th}}$ entry is $(A_1)_{i_1i_1'}\dots (A_k)_{i_ki_k'}$, where $\phi(n)=(i_1,\dots,i_k)$ and $\phi(m)=(i_1',\dots,i_k')$.\\
	To illustrate this, we consider the following example. For $k=2$, $N_1=N_2=2$, $A_1=\begin{pmatrix}
	1&1\\
	1&0
	\end{pmatrix}$, $A_2=\begin{pmatrix}
	0&1\\
	1&1
	\end{pmatrix}$, we have $N=4$ and $\phi(0)=(0,0)$, $\phi(1)=(0,1)$, $\phi(2)=(1,0)$, and $\phi(3)=(1,1)$. Thus 
	
	\begin{eqnarray*}
		A&=&\begin{pmatrix}
			& \textbf{00}&\textbf{01}&\textbf{10}&\textbf{11}\\
			\textbf{00}&(A_1)_{00}(A_2)_{00}&(A_1)_{00}(A_2)_{01}&(A_1)_{01}(A_2)_{00}&(A_1)_{01}(A_2)_{01}\\
			\textbf{01}&(A_1)_{00}(A_2)_{10}&(A_1)_{00}(A_2)_{11}&(A_1)_{01}(A_2)_{10}&(A_1)_{01}(A_2)_{11}\\
			\textbf{10}&(A_1)_{10}(A_2)_{00}&(A_1)_{10}(A_2)_{01}&(A_1)_{11}(A_2)_{00}&(A_1)_{11}(A_2)_{01}\\
			\textbf{11}&(A_1)_{10}(A_2)_{10}&(A_1)_{10}(A_2)_{11}&(A_1)_{11}(A_2)_{10}&(A_1)_{11}(A_2)_{11}
		\end{pmatrix}\\
		&=&\begin{pmatrix}
			(A_1)_{00}A_2&(A_1)_{01}A_2\\
			(A_1)_{10}A_2&(A_1)_{11}A_2
		\end{pmatrix} \\
		&=& A_1\otimes A_2\\
		&=&\begin{pmatrix}
			0&1&0&1\\
			1&1&1&1\\
			0&1&0&0\\
			1&1&0&0
		\end{pmatrix}.
	\end{eqnarray*}
	In general, $A=A_1\otimes A_{2}\otimes\dots\otimes A_k$. A direct argument using induction is as follows. Let $A^\prime$ be the $N^\prime\times N^\prime$ transition matrix corresponds to the product map $f_1\times \dots\times f_{k-1}$ where $N^\prime=N_1N_2\dots N_{k-1}$. By induction hypothesis, $A^\prime=A_{1}\otimes\dots\otimes A_{k-1}$. Let $\phi^\prime$ be the bijection from  $\{0,1,\dots, N^\prime-1\}$ to $\{(i_1,\dots,i_{k-1}) \ \vert \ 0\leq i_j\leq N_j-1,\ 1\leq j\leq k-1\}$. For $0\leq s,t\leq N^\prime-1$, let $\phi^\prime(s)=(i_1,i_2,\dots,i_{k-1})$ and $\phi^\prime(t)=(i_1',i_2'\dots,i_{k-1}')$. Hence $\text{st}^{th}$ entry of $A^\prime$ is $(A_1)_{i_1i_1'}\dots (A_{k-1})_{i_{k-1}i_{k-1}'}$. For a given $0\leq m,n\leq N-1$, where $N=N^\prime N_k$, let $\phi(m)=(i_1,\dots,i_{k-1},i_k)$ and $\phi(n)=(i_1',\dots,i_{k-1}',i_k')$. \\
	The $mn$-th entry of $A$ is given by $(A_1)_{i_1i_1'}\dots (A_k)_{i_ki_k'}=(A^\prime)_{st}(A_k)_{i_ki_k'}$, where $\phi^\prime(s)=(i_1,i_2,\dots,i_{k-1})$ and $\phi^\prime(t)=(i_1',i_2',\dots,i_{k-1}')$.
	Hence, similar to $k=2$ case, note that
	
	\begin{eqnarray*}
		A&=&
		\begin{pmatrix}
			(A^\prime)_{00}A_k&(A^\prime)_{01}A_k&\dots&(A^\prime)_{0(N_k-1)}A_k\\
			(A^\prime)_{10}A_k&(A^\prime)_{11}A_k&\dots&(A^\prime)_{1(N_k-1)}A_k\\
			\vdots&\vdots&\vdots&\vdots\\
			(A^\prime)_{(N_k-1)0}A_k&(A^\prime)_{(N_k-1)1}A_k&\dots&(A^\prime)_{(N_k-1)(N_k-1)} A_k
		\end{pmatrix}\\
		&=&A^\prime\otimes A_k\\
		&=&A_1\otimes A_{2}\otimes\dots\otimes A_k.
	\end{eqnarray*}
	
	\noindent 3) If each $A_j$ is irreducible, then $A$ is irreducible. This follows from the tensor product representation of $A$ in terms of $A_1,\dots, A_k$. The matrix $A$ is irreducible means there exists a finite word that begins with $m$ and ends with $n$, for any given $m,n\in\{0,1,\dots,N-1\}$. \\
	4) With this transition matrix $A=A_1\otimes A_2\otimes\dots\otimes A_k$, we can define a (Parry) measure $\mu$ on $\Sigma_A$ as before. This gives a measure $\tilde{\mu}$ on $I^k$ such that the map $\pi$ in Theorem~\ref{product-factor} is a conjugacy. Let $\mu_s$ be the (Parry) measure on $\Sigma_{A_s}$ and $\tilde{\mu_s}$ be the corresponding measure on $I$ such that each map $\pi_s:\Sigma_{A_s}\to I$ is a conjugacy between $f_s$ on $I$ and $\sigma$ on $\Sigma_{A_s}$ as in Theorem~\ref{thm:symbonto} for $s=1,2,\dots,k$. These measures give a product measure $\tilde{\mu_1}\times\tilde{\mu_2}\times\dots\times\tilde{\mu_k}$ on $I^k$. We will show that $\tilde{\mu}=\tilde{\mu_1}\times\tilde{\mu_2}\times\dots\times\tilde{\mu_k}$. 
	
	Let $w=i_1i_2\dots i_n$ be an allowed word in $\Sigma_A$. Let $\phi(i_j)=(i_j^1,i_j^2,\dots,i_j^k)$ for $1\leq j\leq n$. Then $w^s=i_1^si_2^s\dots i_n^s$ is an allowed word in $\Sigma_{A_s}$. Note that if $\lambda_1,\lambda_2,\dots,\lambda_k$ are the largest eigenvalues of $A_1,A_2,\dots,A_k$ with right and left normalized eigenvectors $v^1,v^2,\dots,v^k$ and $u^1,u^2,\dots,u^k$  respectively. Then by the properties of the tensor product of matrices, $\lambda=\lambda_1\lambda_2\dots\lambda_k$ is the largest eigenvalue of $A$ with right and left eigenvectors $v=v^1\otimes v^2\otimes\dots\otimes v^k$ and $u=u^1\otimes u^2\otimes\dots\otimes u^k$, respectively. Hence it is easy to see that 
	\[
	\mu(C_w)=\mu_1(C^1_{w^1})\mu_2(C^2_{w^2})\dots\mu_k(C^k_{w^k}),
	\]
	where $C^s_{w^s}$ denotes the cylinder based at $w^s$ in $\Sigma_{A_s}, 1\leq s\leq k$.
\end{rem}

\section{Escape rate for product of maps} \label{sec:prod}
In this section we consider product of expansive Markov maps and discuss the escape rates into different holes. All the notations will be same as in Section~\ref{sec:prelim}. Let $q=N=N_1N_2\dots N_k$. 

For $1\leq s\leq k$, let $w^s=i_1^si_2^s\dots i_{n_s}^s$ be a word of length $n_s$ that is allowed in $\Sigma_{A_s}$, and let $I^s_{w^s}$ denote the interval in $I$ corresponding to the cylinder $C^s_{w^s}$, which is known as a \textit{basic interval}. We will consider a hole $H$ of the following type: 
\[
H=I^1_{w^1}\times I^2_{w^2}\times\dots\times I^k_{w^k}.\]
If $n_1=n_2=\dots=n_k=n$, then consider the word $w=i_1i_2\dots i_n$ where $\phi(i_j)=(i_j^1,i_j^2,\dots,i_j^k)$, $1\leq j\leq n$. Note that in this case, through the conjugacy $\pi$, the hole $H$ corresponds to the cylinder based at $w$ in $\Sigma_A$. Let $R_w$ denote the rectangle corresponds to the cylinder based at $w$ in $\Sigma_A$, known as the \textit{basic rectangle}. Hence in this case, we have $H=R_w$. 

Otherwise, choose $n=\max\{n_1,\dots,n_k\}$ and write $C^s_{w^s}=\bigcup_{|u|=n-n_s}C^s_{w^su}$. For example, consider full shift with symbols from $\{0,1,2\}$, then $C_{01}=C_{010}\cup C_{011}\cup C_{012}$. Now, each of these basic intervals correspond to a union of cylinders based at words of equal length $n$. Hence in this case, $H=\bigcup_w R_w$ where $w$'s are obtained from these new collections of words of length $n$. Thus, in general, a product of basic intervals gives a union of basic rectangles. 

\subsection{Transition matrix corresponding to the forbidden words} \label{sec:matrix}
Let $q\geq 2$ and $\Lambda=\{0,1,\dots,q-1\}$ be a collection of symbols (alphabets). Let $\Sigma=\Sigma_q^+$ be the collection of all one-sided sequences with symbols from $\Lambda$. Let $\mathcal{F}=\{w_1,w_2,\dots,w_t\}$ be a collection of words of same length $n\geq2$ with symbols from $\Lambda$ (if words are of different length, we can replace them by a collection of words of length  $n$ such that the associated subshift of  finite type remains the same). Let $\Sigma_\mathcal{F}$ be the collection of all sequences in $\Sigma$ which do not contain words from a collection $\mathcal{F}$ (that is, words in $\mathcal{F}$ are forbidden in sequences in $\Sigma_\mathcal{F}$). The collection $\Sigma_\mathcal{F}$ is called an \textit{\mbox{$(n-1)$}-step shift of finite type}. \\
Define the \textit{transition matrix} $A$ as a $q^{n-1}\times q^{n-1}$ matrix with entries 0 or 1 constructed as below.\\
When $n=2$, $A_{ij}=0$ if and only if the word $ij\in\mathcal{F}$ (the rows and columns of $A$ are labelled from 0 to $q-1$). When $n>2$, let $\Lambda_n=\{u_1,u_2,\dots,u_{q^{n-1}}\}$ be all the words with symbols from $\Lambda$ of length $n-1$. We will now treat each $u_i$ as a symbol, and hence a total of $q^{n-1}$ symbols. Let $\Sigma_n=\Lambda_n^\mathbb{N}$. Let $u_i=a_1^ia_2^i\dots a_{n-1}^i$. We say that the word $u_iu_j$ of length 2 with symbols from $\Lambda_n$ is allowed if and only if $a_{l+1}^i=a_{l}^j$ for every $1\leq l \leq n-2$ and the word $a_1^ia_2^i\dots a_{n-1}^ia_{n-1}^j\notin \mathcal{F}$. Let $\mathcal{F}'$ be the collection of all words of length two with symbols from $\Lambda_n$ which are not allowed (forbidden). Define the transition matrix $A$ to be the $q^{n-1}\times q^{n-1}$ matrix corresponding to this new collection $\mathcal{F}'$ of forbidden words of length two.\\
If $f(m)$ denotes the number of words of length $m$ in $\Sigma$ that do not contain words from $\mathcal{F}$, then the topological entropy of $\Sigma_\mathcal{F}$ (refer ~\cite[Proposition 3.5]{Pollicott_book}) is given by
\begin{equation}\label{eq:entropy}
	h_{\text{top}}(\Sigma_\mathcal{F})= \lim_{m \to \infty}\frac{1}{m}\ln(f(m))=\ln\left(\lambda_{\text{max}}\right),
\end{equation}
where $\lambda_{\text{max}}>0$ is the largest real (Perron) eigenvalue of $A$ provided $A$ is an irreducible matrix.

\subsection{Maps conjugate to a full shift}\label{sec:fullshift}
Suppose each of the map $f_s$ is conjugate to a full shift $\Sigma_{q_s}^+$, say, for $1\leq s\leq k$, and thus the product map $f=f_1\times f_2\times\dots\times f_k$ is conjugate to the full shift $\Sigma_{q}^+$, where $q=q_1q_2\dots q_k$.

We have seen that, in this case, $\tilde{\mu_s}(I^s_{w^s})=\frac{1}{q_s^{n_s}}$, where $|w^s|=n_s$ ($|w|$ denotes the length of $w$), $1\leq s\leq k$. Generally, let $H=R_{w_1}\cup R_{w_2}\cup\dots\cup R_{w_t}$ and $\tilde{\mu}(R_{w_\ell})=\frac{1}{q^n}$ where $n=|w_1|=|w_2|=\dots=|w_t|$ and $1\leq \ell\leq t$. 

The set $I^k\setminus\Omega_m(H)$ consists of all points in $I^k$ that do not enter the hole $H$ in the first $m$ iterations under $f$. By the conjugacy $\pi:\Sigma_q^+\to I^k$, this corresponds to the set of sequences in $\Sigma_q^+$ that do not enter any of the cylinders $C_{w_1}, C_{w_2},\dots,C_{w_t}$ in the first $m$ iterations under the shift map $\sigma$, that is the sequences in $\Sigma_q^+$ that do not contain the words $w_1,w_2,\dots,w_t$ in its first $m+n$ positions, where $n$ is the length of each of the word. Hence $I^k\setminus\Omega_m(H)$ is the union of rectangles which corresponds to cylinders based at words of length $m+n$ that do not contain $w_1,w_2,\dots,w_t$. Since each of these rectangles have measure $\frac{1}{q^{m+n}}$,   
\[\mu(I^k\setminus\Omega_m(H))=\frac{f(m+n)}{q^{m+n}},\]
where $f(r)$ denote the number of words of length $r$ with symbols from $\Lambda$ that do not contain any of the words $w_1,w_2,\dots,w_t$. Thus, using~\eqref{eq:entropy}, if $\mathcal{G}$ is the collection of words $w_1,w_2,\dots,w_t$, then the escape rate
\begin{eqnarray}
	\notag \rho(H)&=&-\lim_{m \to \infty}\frac{1}{m}\ln 
	\frac{f(m+n)}{q^{m+n}}\\
	\notag &=& \lim_{m \to \infty}\frac{1}{m}\ln q^{m+n}-\lim_{m \to \infty}\frac{1}{m}\ln f(m+n) \\
	&=& \ln(q)-h_{\text{top}}(\Sigma_\mathcal{G}). \label{eq:fullshift}
\end{eqnarray}

\subsection{Maps conjugate to a subshift of finite type} \label{subsec:subshift}
Now we consider the general situation when the maps $f_1,f_2,\dots,f_k$ are conjugate to some subshift of finite type. Hence the product map $f$ is conjugate to a subshift of finite type $\Sigma_M$ with the transition matrix $M$. \\
Let the hole $H\subset I^k$ corresponds to union of cylinders based at some collection of words $\mathcal{F}_1$ of same length $n\geq 2$. Let $\mathcal{F}$ be the collection of all the words of length $n$ that do not appear as sequences in $\Sigma_M$. Then $\Sigma_M=\Sigma_{\mathcal{F}}$.
Further we can assume that $\mathcal{F}\cap\mathcal{F}_1=\emptyset$. Let $A,B$ denote the $q^{n-1}\times q^{n-1}$ transition matrices corresponding to the forbidden words in $\mathcal{F}$ and $\mathcal{F}\cup\mathcal{F}_1$ as defined earlier in this section. If $w=i_1i_2\dots i_m$ is a word that appear in a sequence in $\Sigma_{\mathcal{F}}$ that does not contain any word from $\mathcal{F}_1$ as subword, then  
\[
\mu(C_w)=\frac{u_{i_1}v_{i_m}}{\lambda^{m-1}}a_{i_1i_2}a_{i_2i_3}\dots a_{i_{m-1}i_m}=\frac{u_{i_1}v_{i_m}}{\lambda^{m-1}}b_{i_1i_2}b_{i_2i_3}\dots b_{i_{m-1}i_m},
\]
where $\lambda$ is the largest eigenvalue of $A$ and $v = (v_0,v_1,\dots,v_{q^{n-1}-1})^T$, and $u = (u_0,u_1,\dots,u_{q^{n-1}-1})$ are the normalized right and left eigenvectors of $A$ with respect to the eigenvalue $\lambda$ such that $uv=1$. Hence 
\begin{eqnarray*}
	\tilde{\mu}(I^k\setminus \Omega_m(H))&=&\sum_{w=i_1i_2\dots i_m} \frac{u_{i_1}v_{i_m}}{\lambda^{m-1}}b_{i_1i_2}b_{i_2i_3}\dots b_{i_{m-1}i_m}\\
	&=&\sum_{i,j=0}^{q^{n-1}-1}\frac{u_iv_j}{\lambda^{m-1}}\sum_{i_2i_3\dots i_{m-1}}b_{ii_2}b_{i_2i_3}\dots b_{i_{m-1}j}\\
	&=&\sum_{i,j=0}^{q^{n-1}-1}\frac{u_iv_j}{\lambda^{m-1}}(B^{m-1})_{ij}.
\end{eqnarray*}

\noindent Thus we have the following result.

\begin{thm} \label{thm:subshift}
	With notations as above,
	\begin{equation}\label{eq:subshift}
		\rho(H)= h_{\text{top}}(\Sigma_{\mathcal{F}}) - h_{\text{top}}(\Sigma_{\mathcal{F}\cup \mathcal{F}_1}).
	\end{equation}
\end{thm}

\begin{proof}
	Since $u, v$ are positive vectors and $uv=1$, there exists $\alpha>0$ such that $\alpha\leq u_iv_j\leq 1$ for all $0\leq i,j\leq q^{n-1}-1$. Hence
	\[
	-\lim_{m \to \infty}\frac{1}{m}\ln\left(\frac{\sum_{i,j=0}^{q^{n-1}-1}(B^{m-1})_{ij}}{\lambda^{m-1}}\right)\leq \rho(H)\leq -\lim_{m \to \infty}\frac{1}{m}\ln\left(\alpha\frac{\sum_{i,j=0}^{q^{n-1}-1}(B^{m-1})_{ij}}{\lambda^{m-1}}\right),
	\]
	which implies $\rho(H)=h_{\text{top}}(\Sigma_{\mathcal{F}})- h_{\text{top}}(\Sigma_{\mathcal{F}\cup \mathcal{F}_1})$.
\end{proof}
\begin{rem}
	In the case of full shift, we have $\mathcal{F}=\phi$ and thus $\Sigma_{\mathcal{F}}=\Sigma_q^+$. Hence we recover~\eqref{eq:fullshift} since $h_{\text{top}}(\Sigma_q^+)=\ln(q)$.
\end{rem}

\noindent Now we discuss some results from combinatorics to understand $f(m)$.

\subsection{Results from combinatorics}\label{sec:comb}

We refer to~\cite{Combinatorial} for details.

\begin{Def}
	We call a collection of words $\{w_1,w_2,\dots,w_k\}$ with symbols from $\Lambda$ \textit{reduced} if $w_i$ is not a subword of $w_j$ for any $i,j\in\{1,2,\dots,k\}$. 
\end{Def}

\begin{Def}
	Let $u$ and $w$ be two words from some symbol set $\Lambda$ of length $n_1$ and $n_2$, respectively. The \textit{correlation function} of $u$ and $w$, $corr(uw)=[b_1,b_2,\dots,b_{n_1}]$ is defined as follows.\\	
	Place a copy of the word $w$ under $u$ and shift it to the right by $\ell$ digits. If the overlapping parts match then $b_{\ell+1}=1$, otherwise $b_{\ell+1}=0$. The \textit{correlation polynomial} is defined as $(uw)_z=b_1.z^{n_1-1}+b_2.z^{n_1-2}+\dots+b_{n_1}$.\\
	When $u\ne w$, $corr(uw)$ and $corr(wu)$ are called \textit{cross-correlations}. Otherwise $corr(ww)$ is said to be the \textit{autocorrelation} of $w$ and is denoted as $corr(w)$, and the corresponding polynomial is called the \textit{autocorrelation polynomial}.
\end{Def}

\begin{exam}
	We give an example to illustrate the concept of correlation polynomial. Consider two words $u=101001$ $w=10010$, then we have
	\begin{center}
		\begin{tabular}{c c c c c c c c c c c }
			$\ell$&&	\textbf{1} &\textbf{0} &\textbf{1} & \textbf{0} & \textbf{0} & \textbf{1}  &&& $b_{\ell+1}$ \\ 
			0&&   1&0&0&1&0& &&&0\\
			1&&	  & 1 &0 & 0 & 1 &0  &&&0\\ 
			2&&	  &&\textbf{1}&\textbf{0}&\textbf{0}&\textbf{1} &&&1\\
			3&&	  &&&1&0&0  &&&0\\
			4&&	  &&&&1&0  &&&0\\
			5&&	  &&&&&\textbf{1} &&&1
		\end{tabular}
	\end{center}
	We get $corr(uw)=001001$ and $(uw)_z=z^3+1$. Similarly $corr(wu)=00010$. Hence in general, $corr(uw)\neq corr(wu)$. Also note that $corr(u)=100001$ and $corr(w)=10010$.
\end{exam}

\begin{rem}
	Let $C_u$ and $C_w$ be the cylinders based at $u$ and $w$, respectively in $\Sigma_q^+$. Let $n$ be the length of $u$. Then $(uw)_z = \sum_{\ell=1}^{n} b_\ell z^{n-\ell}$, where $b_{\ell}=0$, if $\sigma^{\ell-1}(C_u)\cap C_w=\emptyset$ and $b_{\ell}=1$, if $\sigma^{\ell-1}(C_u)\cap C_w\ne\emptyset$. Thus, if $\sigma^{\ell-1}(C_u)\cap C_w=\emptyset$, for all $1\leq \ell\leq n$, then $(uw)_z=0$. Also, $(uu)_z$ is always a monic polynomial with degree $n-1$. 
\end{rem}

\noindent Let $\{w_1,w_2,\dots,w_k\}$ be a reduced collection of words with symbols from $\Lambda$. Let $f(n)$ denote the number of words of length $n$ with symbols from $\Lambda$ which do not contain any of the words $w_1,w_2,\dots,w_k$. The corresponding generating function is given by
\begin{eqnarray*}
	F(z)&:=&\sum_{n=0}^\infty f(n)z^{-n}.
\end{eqnarray*}
For $i=1,2,\dots,k$, let $f_i(n)$ denote the number of words of length $n$ with symbols from $\Lambda$ that end with the word $w_i$, and do not contain any of $w_1,w_2,\dots,w_k$ except for a single appearance of $w_i$ at the end. Let $F_i(z)$ denote the corresponding generating function for $f_i$. That is,
\[
F_i(z)=\sum_{n=0}^\infty f_i(n)z^{-n}.
\]

\noindent The following theorem is a result in combinatorics, given in~\cite[Theorem 1]{Combinatorial}.

\begin{thm} \label{thm:generating}
	For a reduced collection of words $w_1,w_2,\dots,w_k$, the generating function $F(z),F_1(z),\dots, F_k(z)$ satisfy the following system of linear equations
	\begin{eqnarray*}
		(z-q)F(z)+zF_1(z)+zF_2(z)+\dots+z F_k(z)&=& z\\
		F(z)-z(w_1w_1)_zF_1(z)-z(w_2w_1)_zF_2(z)-\dots-z(w_kw_1)_zF_k(z)&=& 0\\
		\vdots\\
		F(z)-z(w_1w_k)_zF_1(z)-z(w_2w_k)_zF_2(z)-\dots-z(w_kw_k)_zF_k(z)&=& 0,
	\end{eqnarray*}
	where $(w_iw_j)_z$ denotes the correlation polynomial of $w_i$ and $w_j$ for all $i,j=1,2,\dots,k$.
\end{thm}

\begin{rem}\label{rem:F}
	1) Let $P$ be the matrix corresponding to the linear system of equations given in Theorem~\ref{thm:generating}, that is,
	\[
	P=
	\begin{bmatrix}
	z-q&z&z&\dots&z\\
	1&-z(w_1w_1)_z&-z(w_2w_1)_z&\dots&-z(w_k w_1)_z\\
	1&-z(w_1 w_2)_z&-z(w_2w_2)_z&\dots&-z(w_k w_2)_z\\
	\vdots&\vdots&\vdots&\vdots&\vdots\\
	1&-z(w_1w_k)_z&-z(w_2w_k)_z&\dots&-z(w_kw_k)_z
	\end{bmatrix}.
	\]
	\noindent Observe that $P$ is non-singular since the diagonal terms have higher order than the other terms. Hence the generating functions $F(z),F_1(z),\dots, F_k(z)$ are uniquely determined by Theorem~\ref{thm:generating}.\\
	2) For $k=1$ in Theorem~\ref{thm:generating}, and $w=w_1$, we get
	\[
	F(z)=\frac{z}{(z-q)+\dfrac{1}{(ww)_z}},\  \text{ and }\ \ F_1(z)=\frac{1}{1+(z-q)(ww)_z}.
	\]
	3) When $k=2$ in Theorem~\ref{thm:generating}, we get 
	\begin{eqnarray} \label{length2}
		F(z)&=& \frac{z}{(z-q)+\frac{(w_1w_2)_z+(w_2w_1)_z-(w_1w_1)_z-(w_2w_2)_z}{(w_1w_2)_z(w_2w_1)_z-(w_1w_1)_z(w_2w_2)_z}}.
	\end{eqnarray}
	In general, the form of the generating function $F(z)$ is described in the next theorem.
\end{rem} 

\begin{thm}\label{thm:formF}
	With notations as above, we have
	\[
	F(z)=\frac{z}{(z-q)+a(z)},
	\]
	where $a(z)$ is the sum of entries of the matrix $M^{-1}$ where $M$ is the \textit{correlation matrix} given by
	\begin{eqnarray}\label{matrixM}
		\begin{bmatrix}
			(w_1w_1)_z&(w_2w_1)_z&\dots&(w_k w_1)_z\\
			(w_1 w_2)_z&(w_2w_2)_z&\dots&(w_k w_2)_z\\
			\vdots&\vdots&\vdots&\vdots\\
			(w_1w_k)_z&(w_2w_k)_z&\dots&(w_kw_k)_z\\
		\end{bmatrix}
	\end{eqnarray}
\end{thm}
\begin{proof}
	We have
	\begin{eqnarray*}
		\begin{bmatrix}
			F(z)\\F_1(z)\\ \vdots\\ F_k(z)
		\end{bmatrix}
		&=& P^{-1}
		\begin{bmatrix}
			z\\0\\ \vdots \\0
		\end{bmatrix}.
	\end{eqnarray*}
	Thus $F(z)=z(P^{-1})_{11}$. Let $\tilde{P}_{ij}$ denote the $ij$-th minor of $P$, then 
	\begin{eqnarray*}
		(P^{-1})_{11}&=&\frac{1}{\text{det} P} \text{det} \tilde{P}_{11}= \frac{1}{\text{det} P} (-z)^k \text{det} M,
	\end{eqnarray*} 
	where $M$ is the $k\times k$ matrix given in~(\ref{matrixM}). Hence $F(z)=z(-z)^k \frac{\text{det} M}{\text{det} P}$. So the result holds if and only if 
	\begin{eqnarray}\label{eqn:1}
		\frac{z}{(z-q)+a(z)}=z(-z)^k \frac{\text{det} M}{\text{det} P} \iff (-z)^k(z-q)+(-z)^ka(z)=\frac{\text{det} P}{\text{det} M}.
	\end{eqnarray}
	Note that 
	\[
	\text{det} P=(z-q)(-z)^k \text{det} M+z\sum_{j=2}^{k+1}(-1)^{j+1}\text{det} \tilde{P}_{1j},
	\] 
	which gives
	\begin{eqnarray}\label{eqn:2}
		\frac{\text{det} P}{\text{det} M}=(z-q)(-z)^k+z\frac{\sum_{j=2}^{k+1}(-1)^{j+1}\text{det} \tilde{P}_{1j}}{\text{det} M}.
	\end{eqnarray}
	
	\noindent From~\eqref{eqn:1} and~\eqref{eqn:2}, we need to show that 
	\begin{eqnarray}\label{eqn:3}
		a(z)=z\frac{\sum_{j=2}^{k+1}(-1)^{j+1}\text{det} \tilde{P}_{1j}}{(-z)^k \text{det} M}.
	\end{eqnarray}
	Observe that 
	\[
	\text{det} \tilde{P}_{1j}=(-z)^{k-1}\sum_{i=1}^{k}(-1)^{i+1}\text{det} \tilde{M}_{i(j-1)},
	\] 
	for every $2\leq j\leq k+1$. Hence by change of variable $j\to j-1$
	\[
	\sum_{j=1}^{k}(-1)^{j+1}\text{det} \tilde{P}_{1(j-1)}=(-z)^{k-1}\sum_{i,j=1}^{k} (-1)^{i+j} \text{det} \tilde{M}_{ij},
	\]
	and thus 
	\begin{eqnarray*}
		a(z)&=&\frac{1}{\text{det} M} \sum_{i,j=1}^{k}(-1)^{i+j}\text{det} \tilde{M}_{ij} = z\frac{\sum_{j=2}^{k+1}(-1)^{j+1}\text{det} \tilde{P}_{1j}}{(-z)^k \text{det} M}.
	\end{eqnarray*}
	This proves~\eqref{eqn:3} and hence the result follows.
	
\end{proof}
\begin{rem}
	1) From Theorem~\ref{thm:formF}, it is clear that the generating function (hence the escape rate) depends only on the rational function $a(z)$ of correlation polynomials and the size $q$, of the collection of symbols. Thus, if the associated rational function $a(z)$ is the same for two reduced collections $\mathcal{G}_1$ and $\mathcal{G}_2$ of words with symbols from $\Lambda$, then from~\eqref{eq:fullshift}, the escape rates into the holes in $I^k$ corresponding to $\cup_{u\in \mathcal{G}_1}C_u$ and $\cup_{w\in \mathcal{G}_2}C_w$ are the same. \\
	2) We would like to emphasize that the combinatorial  approach described in Section~\ref{sec:comb} is not useful when the number of forbidden words is large. In that case, transition matrix techniques described in Section~\ref{sec:matrix} and~\ref{sec:fullshift} are useful.
\end{rem}
\section{An example of product map on $\mathbb{T}^2$} \label{sec:torus_map}
\noindent Consider the torus $\mathbb{T}^2=S^1\times S^1$, where $S^1\sim [0,1]$ with end-points identified. 
%Let $\lambda$ denote the Haar measure on the interval $S^1$. Let $\mu=\lambda\times \lambda$ be the product measure on $\mathbb{T}^2$ (which is defined as $\mu(B_1\times B_2)=\lambda(B_1)\lambda(B_2)$ on rectangle $B_1\times B_2\subseteq \mathbb{T}^2$ and then approximating it to any Borel set $\mathcal{B}$ using regularity).  \\
\noindent For integers $M,N\geq2$, define the map $T=T_{M,N}:\mathbb{T}^2\to\mathbb{T}^2$ as 
\begin{eqnarray}\label{eq:map}
	T(x,y) &=& (Mx,Ny)\ \text{mod}\ \mathbb{Z}^2.
\end{eqnarray}
Note that $T=T_M\times T_N$, where $T_m:\mathbb{T}\to\mathbb{T}$ defined as 
\begin{eqnarray*}
	T_m(x) &=& mx\ \text{mod}\ \mathbb{Z},
\end{eqnarray*}
is the usual $m$-expanding transformation on the circle which is measure-preserving and ergodic with respect to Lebesgue measure. The map $T_{M,N}$ is measure-preserving and ergodic on the Borel probability measure space $(\mathbb{T}^2,\B,\mu)$. \\
The map $T_m$ is also an expansive Markov map with Markov partition given by $\left\lbrace \left[ \dfrac{i}{m},\dfrac{i+1}{m}\right)\ \vert\ 0\leq i\leq m-1\right\rbrace$.\\~\\ 
We denote $q:=MN$ unless defined otherwise. As before, $\Lambda=\{0,1,\dots,q-1\}$, and $\Sigma_q^+=\Lambda^\mathbb{N}$. The transition matrix of $T_M$ and $T_N$ has all the entries 1 and hence the transition matrix $A$ of $T_{M,N}$ has all the entries 1. Thus as described in Theorem~\ref{product-factor}, $T_{M,N}$ is conjugate to the shift map $\sigma$ on $\Sigma_q^+$. More precisely, consider a sequence $\alpha=\alpha_1\alpha_2\dots \in \Sigma_{q}^+$. For each $i\geq 1$, $\alpha_i=Na_i+b_i$ such that $0\leq b_i\leq N-1$. This can be done uniquely. Since $0\leq\alpha_i\leq q-1$, we get $0\leq a_i\leq M-1$. Let  

\begin{eqnarray}\label{eq:iden}
	x_\alpha &=& \sum_{i=1}^\infty \frac{a_i}{M^i},\ y_\alpha=\sum_{i=1}^\infty \frac{b_i}{N^i}.
\end{eqnarray}

\noindent Note that $a_i \in \{0,1,\dots,M-1\}$, $b_i \in \{0,1,\dots,N-1\}$, for $i\geq 1$. Associate $(x_\alpha,y_\alpha)\in \mathbb{T}^2$ to the sequence $\alpha\in \Sigma_{q}^+$. Define the map $\pi:\Sigma_{q}^+\to\mathbb{T}^2$ as $\pi(\alpha)=(x_\alpha,y_\alpha)$. It defines a conjugacy between $T$ and $\sigma$ since 
\begin{eqnarray*}
	T\circ\pi(\alpha_1\alpha_2\dots)&=&T\left(\sum_{i=1}^{\infty}\frac{a_i}{M^i},\sum_{i=1}^{\infty}\frac{b_i}{N^i}\right) \\
	&=&\left(\frac{a_2}{M}+\frac{a_3}{M^2}+\dots ,\  \frac{b_2}{N}+\frac{b_3}{N^2}+\dots\right) \\
	&=&\pi(\alpha_2\alpha_3\dots) = \pi\circ\sigma(\alpha_1\alpha_2\dots).
\end{eqnarray*} 

\noindent Consider a word $w$ of length $n$ with symbols from $\Lambda$ and let $\mu$ denote the Parry measure on $\Sigma_q^+$. From Remark~\ref{rem:meas}, 3), $\mu(C_w)=\frac{1}{q^n}$ and the induced measure $\tilde{\mu}$ is same as the Lebesgue measure on $\mathbb{T}^2$. Let us denote $\tilde{\mu}=\mu$. As discussed earlier for the full shift case, the measure of a cylinder depends only on the length of the word. Thus to consider the holes with equal measure, it is enough to consider the words of equal length.

\noindent For $m,n\geq 1$, $0\leq i\leq M^m-1$, $0\leq j\leq N^n-1$, set
\[
R_{i,j,m,n}=\left[\frac{i}{M^m},\frac{i+1}{M^m}\right)\times \left[\frac{j}{N^n},\frac{j+1}{N^n}\right),
\]
where the intervals $\left[\frac{i}{M^m},\frac{i+1}{M^m}\right)$ and $\left[\frac{j}{N^n},\frac{j+1}{N^n}\right)$ corresponds to cylinders based at words of length $m$ in $\Sigma_M^+$ and of length $n$ in $\Sigma_N^+$, respectively. \\ 
Note that $\mu(R_{i,j,m,n})=\frac{1}{M^m}.\frac{1}{N^n}$. As discussed earlier, under the conjugacy $\pi$, this rectangle corresponds to a union of cylinders in $\Sigma_{q}^+$. We will explain this in detail later.\\
Let $R_{i,j,m,n}$ and $R_{i^\prime,j^\prime,m^\prime,n^\prime}$ be two rectangles with same measure. Therefore
\begin{equation}\label{eq:area}
	\frac{1}{M^m}.\frac{1}{N^n}=\frac{1}{M^{m^\prime}}.\frac{1}{N^{n^\prime}}\  \Rightarrow\ M^{m^\prime-m}=N^{n-n^\prime}.
\end{equation}
Note that if $m=m'$, then~\eqref{eq:area} holds when $n=n'$. 

\begin{lemma}
	If $m\ne m'$, then~\eqref{eq:area} will hold only when $M^\alpha=N^\beta$, for some $\alpha,\beta\in\mathbb{N}$ with $\gcd(\alpha,\beta)=1$.
\end{lemma}

\begin{proof}
	If~\eqref{eq:area} holds, then $n\ne n'$.\\
	Suppose $M=p_1^{\alpha_1}\dots p_k^{\alpha_k}$ and $N=q_1^{\beta_1}\dots q_\ell^{\beta_\ell}$ be the prime factorizations of $M$ and $N$ with $1<p_1<p_2<\dots <p_k$ and $1<q_1<q_2<\dots <q_\ell$. Hence
	\[
	\left(p_1^{\alpha_1}\dots p_k^{\alpha_k}\right)^{m'-m} = \left(q_1^{\beta_1}\dots q_\ell^{\beta_\ell}\right)^{n-n'}.
	\]
	If $p_j\notin\{q_1,q_2,\dots, q_\ell\}$, then $\alpha_j(m'-m)=0$ which implies $m=m'$, which is a contradiction. Hence $\{p_1,\dots,p_k\}\subseteq \{q_1,\dots, q_\ell\}$. Similarly we can show the reverse containment. Hence $\{p_1,\dots,p_k\}= \{q_1,\dots, q_\ell\}$ and therefore $k=\ell$ and $p_i=q_i$ for all $i$. Thus,
	\[
	\prod_{i=1}^k p_i^{\alpha_i(m'-m)-\beta_i(n-n')} = 1.
	\]
	Thus $\alpha_i(m'-m)-\beta_i(n-n')=0$, for all $i$. Since $m\ne m'$ and $n\ne n'$, we see that $\dfrac{\alpha_i}{\beta_i}=\dfrac{n-n'}{m'-m}=\dfrac{\alpha}{\beta}$, say, for all $i$, where $\gcd(\alpha,\beta)=1$. Hence $M^\alpha=N^\beta$. Further, for all $m, m', n, n'$ with $\beta(n-n')=\alpha(m'-m)$,~\eqref{eq:area} holds.
\end{proof}

\subsection{Computation of escape rate} We will only consider the case when $M^\alpha\ne N^\beta$, for any $\alpha,\beta\in\mathbb{N}$, and thus the rectangles with same measure are given by $R_{i,j,m,n}$ and $R_{i^\prime,j^\prime,m,n}$ for $0\leq i,i^\prime \leq M^m-1$ and $0\leq j,j^\prime\leq N^n-1$. We consider the map $T=T_{M,N}$ on $\mathbb{T}^2$ with a hole given by rectangle of this type (we refer to Definition~\ref{escape-rate}). We look at the following two cases: when $m=n$ and $m>n$ (similar argument works for $m<n$).
\subsubsection{Case 1: $\mathbf{m=n}$}
Under the conjugacy $\pi$ described above, each point in the rectangle $R_{i,j,m,m}$ has a corresponding sequence in $\Sigma_{q}^+$ which starts with a fixed word, say $w$, of length $m$. Thus the rectangle $R_{i,j,m,m}$ corresponds to the cylinder $C_w$ based at $w$ in $\Sigma_{q}^+$ as described in Figure~\ref{fig:1}. That means $R_{i,j,m,m}$ is a basic rectangle. Hence  
\[
\mu(\mathbb{T}^2\setminus\Omega_k(R_w))=\frac{f(k+m)}{q^{k+m}},
\]
where $f(k)$ denote the number of words of length $k$ with symbols from $\Lambda$ that do not contain the word $w$. This implies
\[\rho(R_{i,j,m,m})=\rho(R_w)=-\lim_{k \to \infty}\frac{1}{k}\ln 
\frac{f(k+m)}{q^{k+m}}.\]
This case is same as the one discussed in ~\cite{BY}. Hence we have the following result from~\cite[Theorem 4.5.3]{BY}.

\begin{thm}\label{thm:tau_rho}
	Let $R_{i,j,m,m}$ and $R_{i^\prime,j^\prime,m,m}$ be two holes of the same measure. If $\tau(R_{i,j,m,m})<\tau(R_{i^\prime,j^\prime,m,m})$ then 
	\[
	\rho(R_{i,j,m,m})<\rho(R_{i^\prime,j^\prime,m,m}).\]
\end{thm}

\begin{figure}[h!]
	\centering
	\includegraphics[width=.7\textwidth]{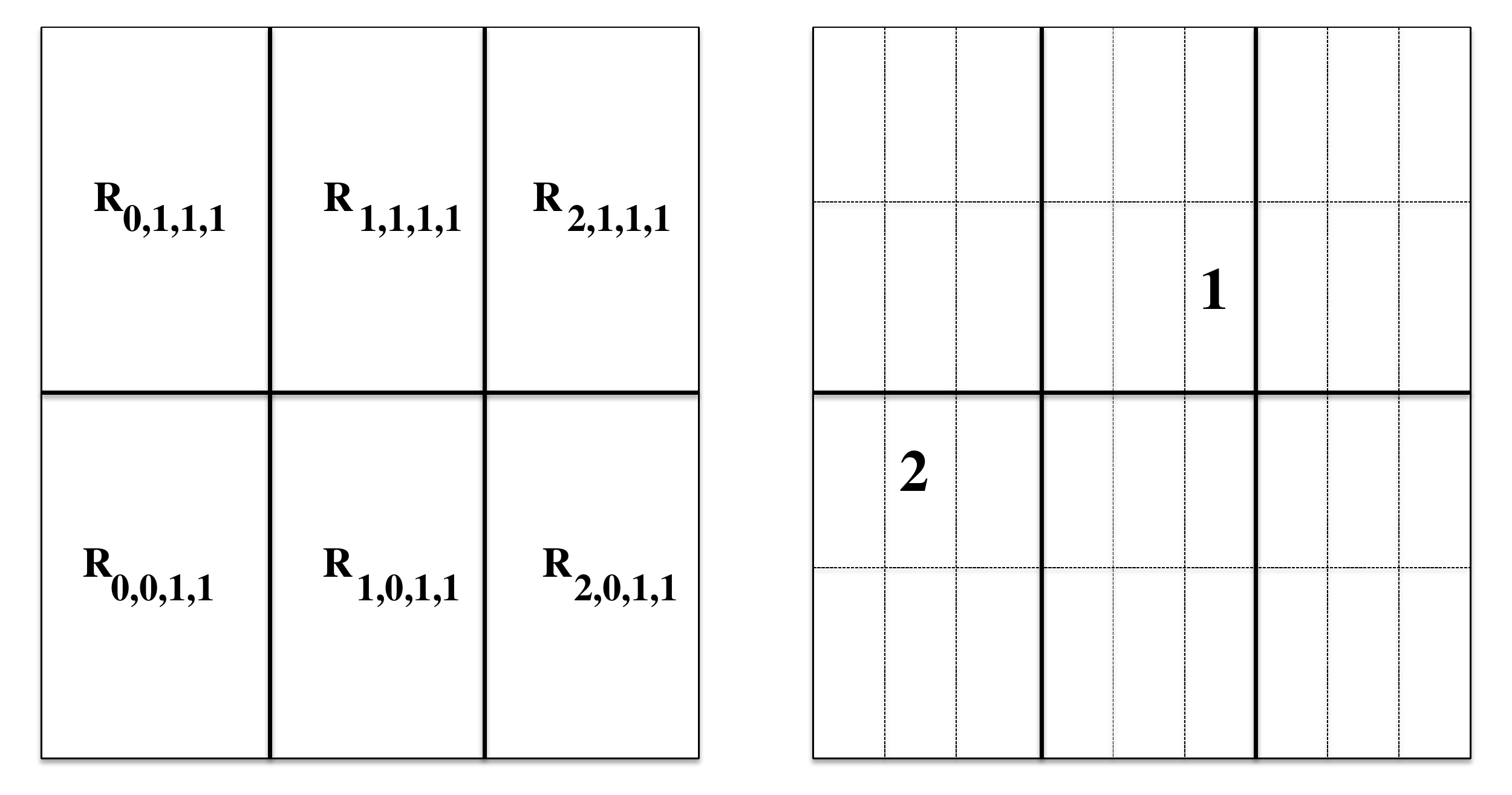}
	\caption{Basic rectangles when $M=3$, $N=2$ and (left) $m=n=1$; $R_{i,j,1,1}\sim C_{2i+j}$ and
		(right) $m=n=2$; (1) $R_{5,2,2,2}\sim C_{34}$, (2) $R_{1,1,2,2}\sim C_{03}$.}
	\label{fig:1}
\end{figure}

\subsubsection{Case 2: $\mathbf{m>n}$}
We write 
\begin{eqnarray*}
	R_{i,j,m,n}&=&\left[\frac{i}{M^m},\frac{i+1}{M^m}\right)\times \left[\frac{j}{N^n},\frac{j+1}{N^n}\right)\\
	&=&\left[\frac{i}{M^m},\frac{i+1}{M^m}\right)\times \left[\frac{j N^{m-n}}{N^m},\frac{jN^{m-n}+1}{N^m}\right) \cup\\
	&&\left[\frac{i}{M^m},\frac{i+1}{M^m}\right)\times \left[\frac{j N^{m-n}+1}{N^m},\frac{jN^{m-n}+2}{N^m}\right)\cup \dots \cup\\
	&&\left[\frac{i}{M^m},\frac{i+1}{M^m}\right)\times \left[\frac{j N^{m-n}+N^{m-n}-1}{N^m},\frac{jN^{m-n}+N^{m-n}}{N^m}\right).
\end{eqnarray*}

\begin{figure}[h!]
	\centering
	\includegraphics[width=0.6\textwidth]{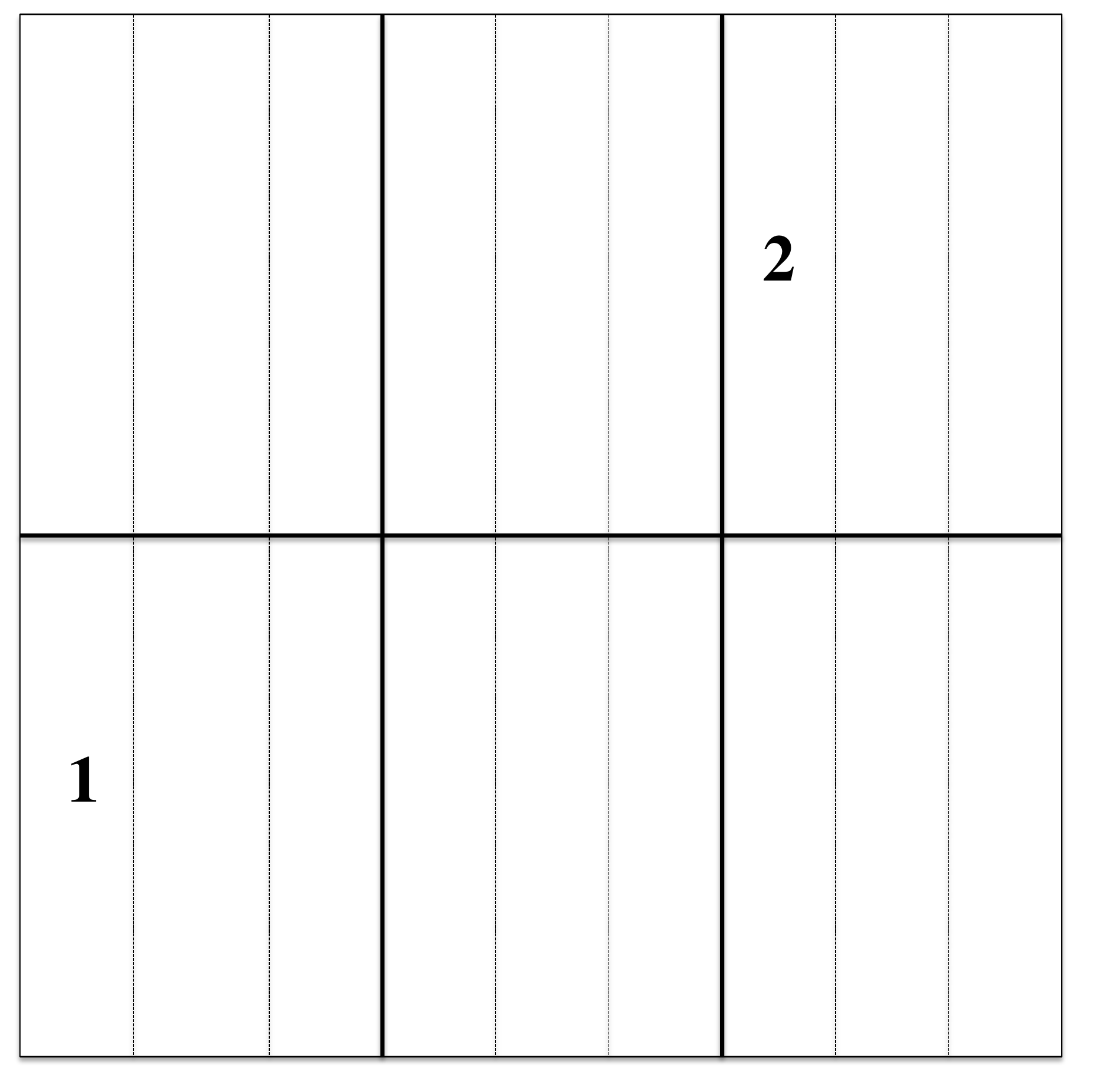}
	\caption{$M=3,N=2,m=2,n=1$, each rectangle is a union of two basic rectangles:\\
		(1) $R_{0,0,2,1}=R_{0,0,2,2}\cup R_{0,1,2,2}\sim C_{00}\cup C_{01}$, \\
		(2) $R_{6,1,2,1}=R_{6,2,2,2}\cup R_{6,3,2,2}\sim C_{50}\cup C_{51}$.}
	\label{fig:aa}
\end{figure}

There are $N^{m-n}$ rectangles in this union, each of which is a basic rectangle of the form $R_{k,l,m,m}$ for some $0\leq k,l\leq M^m-1$. For each rectangle $R_{k,l,m,m}$ we get a corresponding word of length $m$ such that the cylinder based at that word corresponds to the rectangle $R_{k,l,m,m}$. Hence the rectangle $R_{i,j,m,n}$ has a correspondence with the union of $N^{m-n}$ cylinders (each of measure $1/q^m$) given by
\[
C_{uw_1}\cup C_{uw_2}\cup\dots\cup C_{uw_{N^{m-n}}},
\]
where $u$ is a fixed word of length $n$, and $w_1,\dots,w_{N^{m-n}}$ each have length $m-n$. See Figure~\ref{fig:aa} for an example. Similar to previous case,  
\[
\rho(R_{i,j,m,n})=-\lim_{k \to \infty}\frac{1}{k}\ln 
\frac{f(k+m)}{q^{k+m}},
\]
where $f(k)$ is the number of words of length $k$ with symbols from $\Lambda$ that do not contain the words $uw_1,\dots,uw_{N^{m-n}}$.\\~\\
To calculate $f(k)$, we use results from Section~\ref{sec:matrix},~\ref{sec:fullshift} and~\ref{sec:comb}. We will illustrate this with a simple example in Section~\ref{sec:example}. 

\begin{rem}
	Note that the results that will be discussed in Section~\ref{sec:example} and Section~\ref{sec:compare} can be generalized to any map which is conjugate to a full shift. This happens since the properties used from now on do not depend on the geometric nature of the map. However when the product of maps is considered, and hole being the product of basic holes from each component maps, the corresponding holes in symbol space become union of cylinders unlike the case for the individual maps. 
\end{rem}

\subsection{Calculation of escape rate - An example} \label{sec:example}
\noindent Let us consider the map $T=T_{3,2}:\mathbb{T}^2\to \mathbb{T}^2$ given by $T(x,y)=(3x,2y)\ \text{mod}\ \mathbb{Z}^2$.
\begin{exam}
	We consider the case when $m=n$. As described earlier, the rectangle $R_{i,j,m,m}$ corresponds to a cylinder $C_w$. Thus we are looking for sequences where a single word $w$ is forbidden. From Theorem~\ref{thm:formF}, the corresponding generating function $F$ only depends on the autocorrelation of the word $w$. For the following examples, we calculate the largest eigenvalue of the corresponding transition matrix to find the escape rate. See Figure~\ref{aa} for the following two cases. \\
	1) In case $m=2$, each hole will correspond to a cylinder based at a word $w=ab$ of length 2. The escape rate into the holes where $a=b$ will be the same (in this case $(ww)_z=z+1$), and the escape rate into the holes where $a\ne b$ will be the same (in this case $(ww)_z=z$). By Theorem~\ref{thm:tau_rho}, since $\tau(R_{aa})<\tau(R_{ab})$ ($a\ne b$), $\rho(R_{aa})<\rho(R_{ab})$. Here $\rho(R_{aa})=-\ln\left(\dfrac{5+3\sqrt{5}}{12}\right)\sim 0.025$ and $\rho(R_{ab})=-\ln\left(\dfrac{3+2\sqrt{2}}{6}\right)\sim 0.029$.\\
	2) In case $m=3$, each hole will correspond to a cylinder based at a word $w=abc$ of length 3. Three values of escape rate are possible here:
	\begin{enumerate}
		\item[a)] When $a=b=c$ (in this case $(ww)_z=z^2+z+1$).
		\item[b)] When $a=c$ but $a\ne b$ (in this case $(ww)_z=z^2+1$).
		\item[c)] When either $a\ne b$ or $b\ne c$ and also $a\ne c$ (in this case $(ww)_z=z^2$).\\
	\end{enumerate}
	By Theorem~\ref{thm:tau_rho}, since $\tau(R_{aaa})<\tau(R_{aba})<\tau(R_{abc})$ ($a\ne b$, $b\ne c$, $c\ne a$), $\rho(R_{aaa})<\rho(R_{aba})<\rho(R_{abc})$. Here $\rho(R_{aaa})\sim 0.0039, \ \rho(R_{aba})\sim 0.0046$ and $\rho(R_{abc})\sim 0.0047$.
	\begin{figure}[h!]
		\centering
		\includegraphics[width=0.8\textwidth]{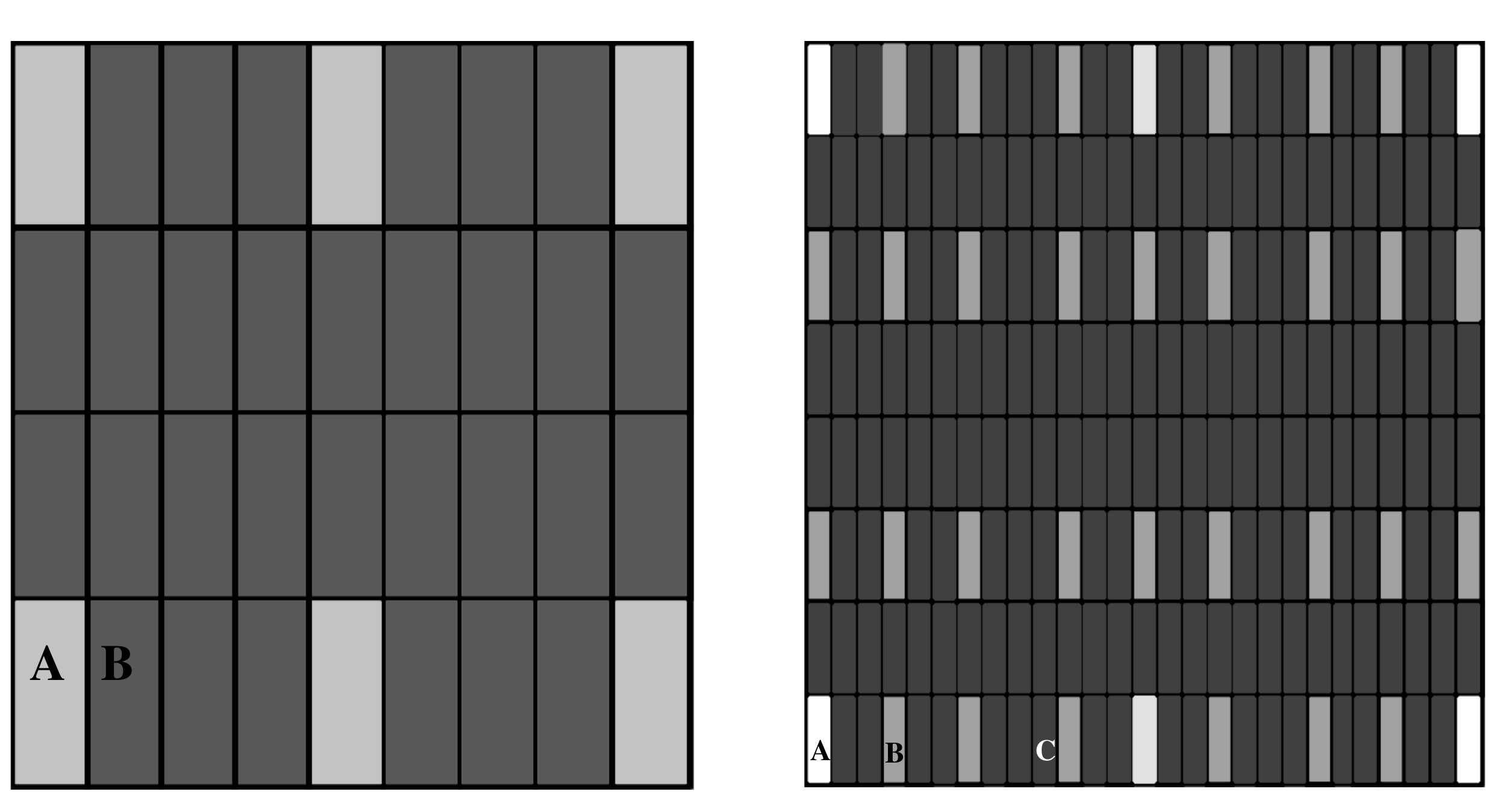}
		\caption{Escape rates into basic rectangles for $M=3, N=2$: (left) $m=n=2$; $\rho(A)\sim 0.025$, $\rho(B)\sim 0.029$;\\
			(right) $m=n=3$; $\rho(A)\sim 0.0039$, $\rho(B)\sim 0.0046$, $\rho(C)\sim 0.0047$.}
		\label{aa}
	\end{figure}
\end{exam}

\begin{exam}
	Let us consider the rectangles $A=\left[0,\frac{1}{9}\right)\times \left[0,\frac{1}{2}\right)$ and $B=\left[\frac{2}{9},\frac{1}{3}\right)\times \left[0,\frac{1}{2}\right)$. These rectangles are of the form $R_{i,j,m,n}$ where $m>n$. \\
	We now calculate the escape rates $\rho(A)$ and $\rho(B)$ using two methods, as described in Section~\ref{sec:matrix} and Section~\ref{sec:comb}. \\
	
	\noindent\textbf{Method 1 (Combinatorics):}
	Since $M=3$ and $N=2$, we have $q=6$ and thus we will consider the collection of one-sided sequences $\Sigma^+_6$ consisting of symbols from the set $\Lambda=\{0,1,\dots,5\}$. 
	
	For the given $A,B$, we have $m=2$ and $n=1$. Hence each rectangle corresponds to a union of two cylinders based at words of length two with the same first letter. We have $A=R_{w_1}\cup R_{w_2}$, where $w_1=00$ and $w_2=01$; and $B=R_{u_1}\cup R_{u_2}$, where $u_1=04$ and $u_2=05$.
	
	Let $f(k)$ denote the number of words of length $k$ with symbols from $\Lambda$, which do not contain the words $w_1$ and $w_2$ and let $F(z)$ denote the corresponding generating function. Similarly let $g(k)$ denote the number of words of length $k$ with symbols from $\Lambda$, which do not contain the words $u_1$ and $u_2$ and let $G(z)$ denote the corresponding generating function. We calculate the following quantities:
	\[
	(w_1w_1)_z=z+1;\ \ (w_1w_2)_z=1;\ \ (w_2w_1)_z= 0;\ \ (w_2w_2)_z=z
	\]
	\[
	(u_1u_1)_z=z;\ \ (u_1u_2)_z=0;\ \ (u_2u_1)_z=0;\ \ (u_2u_2)_z=z.
	\]
	Using~\eqref{length2}, we get
	\[
	F(z)=\frac{z^2+z}{z^2-5z-4}\ \ \text{and} \ \ G(z)=\frac{z^2}{z^2-6z+2}.
	\]
	Now we will calculate the escape rate into $A$. Write 
	\begin{eqnarray*}
		F(z)&=& \frac{z^2+z}{z^2-5z-4}=\sum_{k=0}^\infty f(k)z^{-k}\\
		\implies z^2+z&=& (z^2-5z-4)\left(\sum_{k=0}^\infty f(k)z^{-k}\right).
	\end{eqnarray*}
	Let $f_k=f(k)$. Then comparing the coefficients of $z^{-k}$, we get the following recurrence relation
	\begin{equation} \label{eq:8}
		f_{k+2}=5f_{k+1}+4f_k;\ \ f_0=1;\ \ f_1=6;\ \  k\geq 0.
	\end{equation} 
	We use the standard procedure for solving the linear recurrence relation by taking $f_k=r^k$. Then~\eqref{eq:8} gives the characteristic equation $r^2-5r-4=0$ whose roots are given by $\mu_{\pm}=\frac{5\pm \sqrt{41}}{2}$. Hence the general solution for~\eqref{eq:8} is given by 
	\[
	f_k=c\mu_+^k+d\mu_-^k, \ \ k\geq 0.
	\] 
	Using the initial conditions, $c+d=1$ and $c\mu_++d\mu_-=6$ gives 
	\[
	c=\frac{\mu_--6}{\mu_--\mu_+} \ \ \text{ and }\ \  d=\frac{6-\mu_+}{\mu_--\mu_+}.
	\]
	Hence 
	\[
	f_k=\frac{7\sqrt{41}+41}{82}\mu_+^k+\frac{41-7\sqrt{41}}{82}\mu_-^k,\ \ k\geq 0.
	\]
	Now 
	\begin{eqnarray*}
		\rho(A)&=&-\lim_{k \to \infty}\frac{1}{k}\ln \frac{f(k+2)}{6^{k+2}} = -\lim_{k \to \infty}\frac{1}{k}\ln \frac{\frac{7\sqrt{41}+41}{82}\mu_+^{k+2}+\frac{41-7\sqrt{41}}{82}\mu_-^{k+2}}{6^{k+2}}.
	\end{eqnarray*}
	Since $|\mu_-|<1$, $\ \mu_-^{k+2}$ goes to zero as $k$ tends to infinity. Hence one can neglect the second term. This gives 
	\begin{eqnarray*}
		\rho(A)&=&-\lim_{k \to \infty}\frac{1}{k}\ln \frac{(7\sqrt{41}+41) \mu_+^{k+2}}{82\times 6^{k+2}} = -\lim_{k \to \infty}\frac{1}{k}\ln \frac{\mu_+^{k+2}}{6^{k+2}} \\
		&=& -\lim_{k \to \infty}\frac{1}{k}\ln \frac{\mu_+^k}{6^{k+2}}=-\ln \frac{\mu_+}{6}=-\ln \left(\frac{5+\sqrt{41}}{12}\right).
	\end{eqnarray*}
	
	\noindent Similarly we get $\rho(B)=-\ln \left(\frac{3+\sqrt{7}}{6}\right)$. Note that $\rho(A)<\rho(B)$. Hence we have two holes with the same measure but different escape rates. 
	
	For $M=3,N=2,m=2$ and $n=1$, there are total $M^m.N^n=18$ rectangles of the type $R_{i,j,m,n}$ with measure $\frac{1}{18}$. Following are the corresponding words that represent each of the possible $R_{i,j,2,1}$: 
	\begin{center}
		\begin{tabular}{ c c c c c c }
			00\&01 & 02\&03 & 04\&05 & 10\&11 & 12\&13 & 14\&15  \\ 
			20\&21 & 22\&23 & 24\&25 & 30\&31 & 32\&33 & 34\&35  \\
			40\&41 & 42\&43 & 44\&45 & 50\&51 & 52\&53 & 54\&55.  \\
		\end{tabular}
	\end{center}
	
	Note that for any pair of words in the above list, the correlations are either equal to the correlations obtained for the hole $A$ (this happens when the pair of words are of the form $ab$ and $ac$ where either $a=b$ or $a=c$) or equal to the correlations obtained for the hole $B$ (this happens when the pair of words are $ab$ and $ac$ where both $a\neq b$ and $a\neq c$). Since~\eqref{length2} depends only on the correlation of the words involved, we conclude that escape rate into any hole of the form $R_{i,j,m,n}$ exists for $M=3,N=2,m=2$ and $n=1$, and the value $\rho(R_{i,j,m,n})$ is either $\rho(A)$ or $\rho(B)$.\\
	
	\begin{figure}[h!]
		\centering
		\includegraphics[width=.5\textwidth]{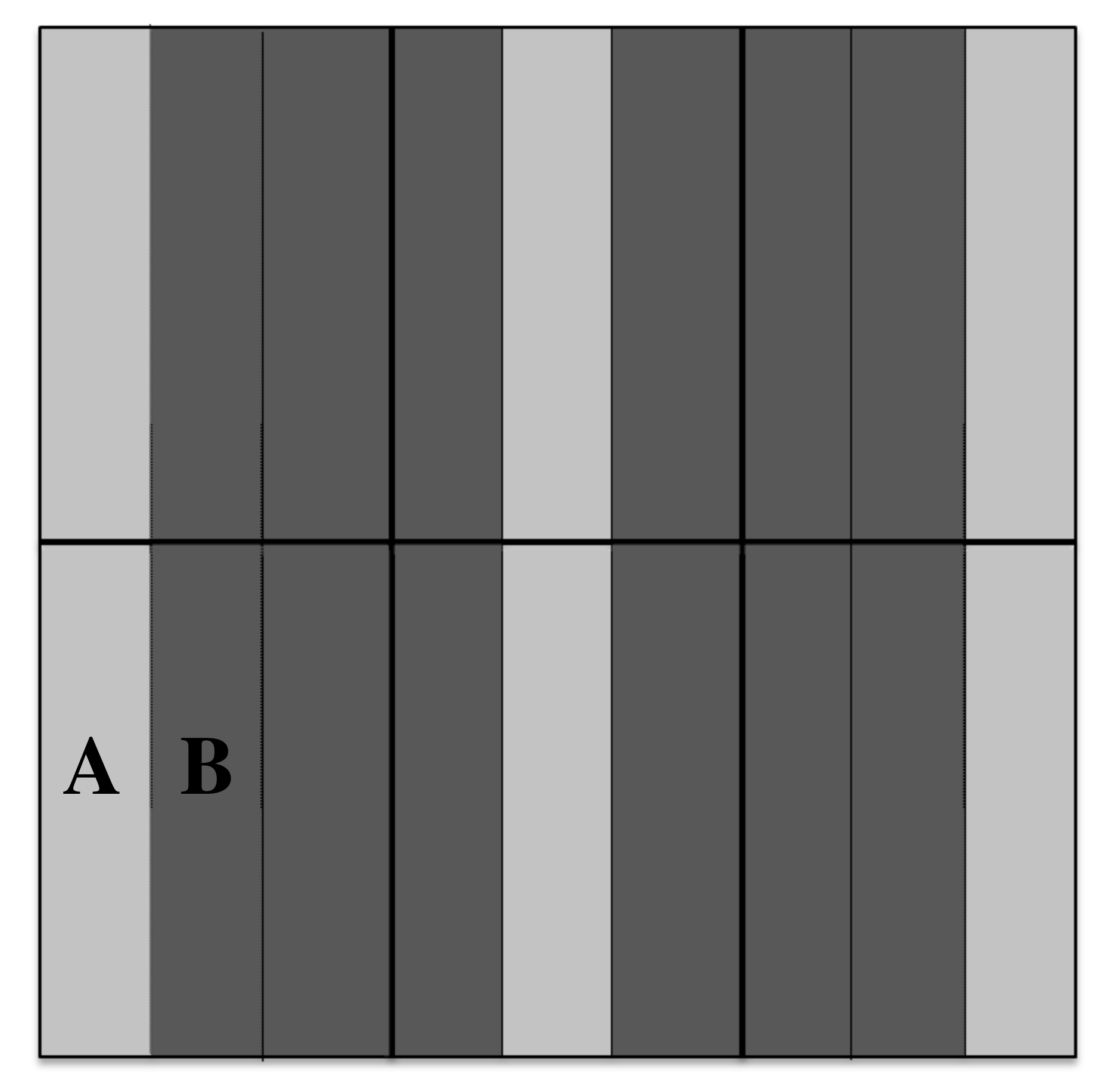}
		\caption{Escape rates of rectangles when $M=3$, $N=2$, $m=2$, $n=1$, $\rho(A)\sim 0.051$, $\rho(B)\sim 0.061$.}
		\label{fig:3}
	\end{figure}
	
	\noindent\textbf{Method 2 (Transition Matrix):}
	Let $P$ and $Q$ denote the transition matrix when the words $w_1=00$, $w_2=01$ are forbidden and when the words $u_1=04$, $u_2=05$ are forbidden, respectively. Then 
	\[
	P= \begin{bmatrix}
	0&0&1&1&1&1\\
	1&1&1&1&1&1\\
	1&1&1&1&1&1\\
	1&1&1&1&1&1\\
	1&1&1&1&1&1\\
	1&1&1&1&1&1\\
	\end{bmatrix}, \ \ Q= \begin{bmatrix}
	1&1&1&1&0&0\\
	1&1&1&1&1&1\\
	1&1&1&1&1&1\\
	1&1&1&1&1&1\\
	1&1&1&1&1&1\\
	1&1&1&1&1&1\\
	\end{bmatrix}
	\]
	which gives $\lambda_{\text{max}}^P=\frac{5+\sqrt{41}}{2}$ and $\lambda_{\text{max}}^Q=3+\sqrt{7}$. Hence $\rho(A)=-\ln \frac{\lambda_{\text{max}}^P}{6}$ and $\rho(B)=-\ln \frac{\lambda_{\text{max}}^Q}{6}$, which gives the same result as calculated using Method 1.
	
	Using the same method the escape rate into any rectangle of the type $R_{i,j,m,n}$ can be computed. Clearly if $m-n$ is large then the number of forbidden words is large, and thus the computation of escape rate becomes harder. In Section~\ref{sec:compare}, we give a particular example where the collection of forbidden words is large and satisfies a certain property.	
\end{exam}

\begin{figure}[h!]
	\centering
	\includegraphics[width=.7\textwidth]{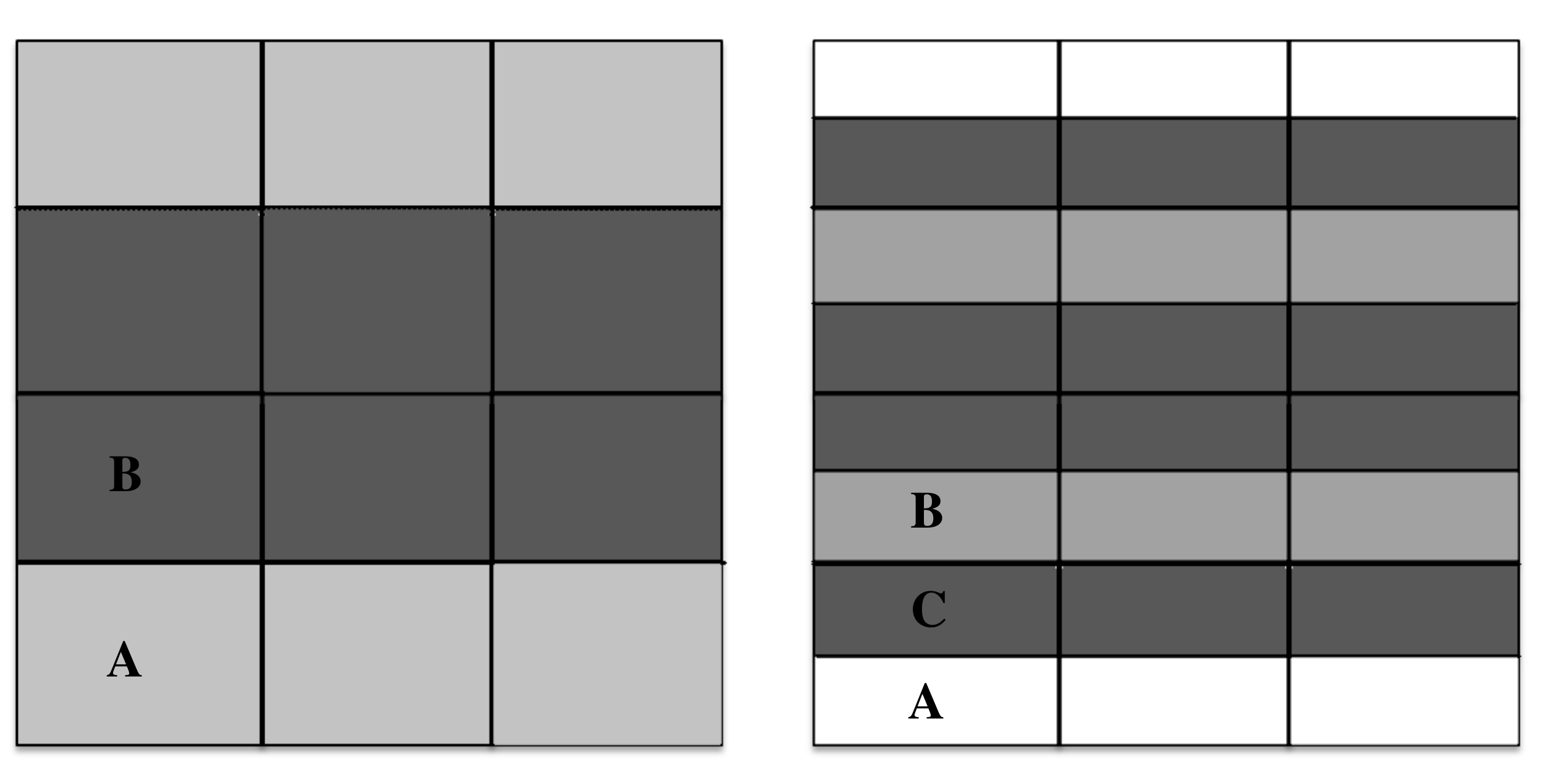}
	\caption{Escape rates for rectangles when \\
		(left) $m=1,n=2$; $\rho(A)\sim 0.08$, $\rho(B)\sim 0.1$; \\
		(right) $m=1,n=3$; $\rho(A)\sim 0.036$, $\rho(B)\sim 0.042$, $\rho(C)\sim 0.047$.}
	\label{fig:4}
\end{figure}

\begin{table}[h!]
	\centering
	\caption{$m=1,n=2$, each rectangle corresponds to union of three cylinders based at words of length two}	
	\begin{tabular}{|c c c c |} 
		\hline
		$R_{i,j,m,n}$ & Corresponding words & $\tau_{\text{min}}$& $\rho(R_{i,j,m,n})\sim$ \\ [0.5ex]
		\hline\hline
		$R_{0,0,1,2}$& $00,02,04$ &1& 0.08  \\\hline
		$R_{0,1,1,2}$& $01,03,05$ &2& 0.1  \\\hline
		$R_{0,2,1,2}$& $10,12,14$ &2& 0.1  \\\hline
		$R_{0,3,1,2}$& $11,13,15$ &1& 0.08  \\\hline
		$R_{1,0,1,2}$& $20,22,24$ &1& 0.08  \\\hline
		$R_{1,1,1,2}$& $21,23,25$ &2& 0.1  \\\hline
		$R_{1,2,1,2}$& $30,32,34$ &2& 0.1  \\\hline
		$R_{1,3,1,2}$& $31,33,35$ &1& 0.08  \\\hline
		$R_{2,0,1,2}$& $40,42,44$ &1& 0.08  \\\hline
		$R_{2,1,1,2}$& $41,43,45$ &2& 0.1  \\\hline
		$R_{2,2,1,2}$& $50,52,54$ &2& 0.1  \\\hline
		$R_{2,3,1,2}$& $51,53,55$ &1& 0.08  \\\hline
		
		\hline
	\end{tabular}
	\label{table:1}
\end{table}

\begin{table} 
	\centering
	\caption{$m=1,n=3$, each rectangle corresponds to union of nine cylinders based at words of length three.}	
	\begin{tabular}{|c c c c|} 
		\hline
		$R_{i,j,m,n}$ & Corresponding words & $\tau_{\text{min}}$& $\rho(R_{i,j,m,n})\sim$ \\ [0.5ex] 
		\hline\hline
		$R_{0,0,1,3}$& $000,002,004,020,022,024,040,042,044$ &1& 0.036 \\\hline 
		$R_{0,1,1,3}$& $001,003,005,021,023,025,041,043,045$ &3& 0.047 \\\hline 	
		$R_{0,2,1,3}$& $010,012,014,030,032,034,050,052,054$ &2& 0.042 \\\hline  
		$R_{0,3,1,3}$& $011,013,015,031,033,035,051,053,055$ &3& 0.047  \\\hline 
		$R_{0,4,1,3}$& $100,102,104,120,122,124,140,142,144$ &3& 0.047  \\\hline 
		$R_{0,5,1,3}$& $101,103,105,121,123,125,141,143,145$ &2& 0.042  \\\hline 
		$R_{0,6,1,3}$& $110,112,114,130,132,134,150,152,154$ &3& 0.047  \\\hline 
		$R_{0,7,1,3}$& $111,113,115,131,133,135,151,153,155$ &1& 0.036 \\\hline

		$R_{1,0,1,3}$& $200,202,204,220,222,224,240,242,244$ &1& 0.036  \\\hline  
		$R_{1,1,1,3}$& $201,203,205,221,223,225,241,243,245$ &3& 0.047 \\\hline 
		$R_{1,2,1,3}$& $210,212,214,230,232,234,250,252,254$ &2& 0.042  \\\hline 
		$R_{1,3,1,3}$& $211,213,215,231,233,235,251,253,255$ &3& 0.047  \\\hline 
		$R_{1,4,1,3}$& $300,302,304,320,322,324,340,342,344$ &3& 0.047  \\\hline 
		$R_{1,5,1,3}$& $301,303,305,321,323,325,341,343,345$ &2& 0.042  \\\hline 
		$R_{1,6,1,3}$& $310,312,314,330,332,334,350,352,354$ &3& 0.047  \\\hline 
		$R_{1,7,1,3}$& $311,313,315,331,333,335,351,353,355$ &1& 0.036  \\\hline

		$R_{2,0,1,3}$& $400,402,404,420,422,424,440,442,444$ &1& 0.036 \\\hline  
		$R_{2,1,1,3}$& $401,403,405,421,423,425,441,443,445$ &3& 0.047 \\\hline 
		$R_{2,2,1,3}$& $410,412,414,430,432,434,450,452,454$ &2& 0.042  \\\hline 
		$R_{2,3,1,3}$& $411,413,415,431,433,435,451,453,455$ &3& 0.047  \\\hline 
		$R_{2,4,1,3}$& $500,502,504,520,522,524,540,542,544$ &3& 0.047  \\\hline 		
		$R_{2,5,1,3}$& $501,503,505,521,523,525,541,543,545$ &2& 0.042  \\\hline
		$R_{2,6,1,3}$& $510,512,514,530,532,534,550,552,554$ &3& 0.047  \\\hline 
		$R_{2,7,1,3}$& $511,513,515,531,533,535,551,553,555$ &1& 0.036  \\\hline
		\hline
	\end{tabular}
	\label{table:2}
\end{table}

\begin{rem}
	Let $M=3$, $N=2$, $m=1$. As indicated in Tables~\ref{table:1} and~\ref{table:2} and Figure~\ref{fig:4}, the escape rate into the rectangles $R_{i,j,1,n}$ when $j$ is fixed is the same. Further, for fixed $i$, say $i=0$, the number of possible values for the escape rates into the collection of holes $\{R_{0,j,1,n}\ \vert\ 0\leq j\leq N^n-1\}$ is same as the number of possible values for the escape rates into the collection of holes $\{I_{j,n}=\left[\frac{j}{N^n},\frac{j+1}{N^n}\right)\ \vert\ 0\leq j\leq N^n-1\}$ for the map $T_N$ on $\mathbb{T}$. Moreover, the order is the same, that is, $\rho(R_{0,j_1,1,n})<\rho(R_{0,j_2,1,n})$ if and only if $\rho_{T_N}(I_{j_1,n})<\rho_{T_N}(I_{j_2,n})$. Hence we get a similar result as in Theorem~\ref{thm:tau_rho}. That means if $\tau_{\text{min}}(R_{0,j_1,1,n})<\tau_{\text{min}}(R_{0,j_2,1,n})$, then $\rho(R_{0,j_1,1,n})<\rho(R_{0,j_2,1,n})$.
	
	\begin{table} 
		\centering
		\caption{$m=4,n=2$, each rectangle corresponds to union of four cylinders based at words of length four.}	
		\begin{tabular}{|c c c c|} 
			\hline
			$R_{i,j,m,n}$ & Corresponding words & $\tau_{\text{min}}$& $\rho(R_{i,j,m,n})\sim$ \\
			\hline\hline
			$R_{0,0,4,2}$ & $0000,0001,0010,0011$ &1& 0.0028\\
			\hline
			$R_{1,0,4,2}$ & $0002,0003,0012,0013$ &4& 0.00312\\
			\hline
			$R_{2,0,4,2}$ & $0020,0021,0030,0031$ &3& 0.00309\\
			\hline
			$R_{10,0,4,2}$& $0202,0203,0212,0213$ &2& 0.00303\\
			\hline
			$R_{0,1,4,2}$ & $0100,0101,0110,0111$ &2& 0.00301\\
			\hline
			$R_{1,1,4,2}$ & $0102,0103,0112,0113$ &4& 0.00312\\
			\hline
			$R_{2,1,4,2}$ & $0120,0121,0130,0131$ &3& 0.00309\\
			\hline
			$R_{10,1,4,2}$& $0302,0303,0312,0313$ &2& 0.00303\\
			\hline
			$R_{0,2,4,2}$ & $1000,1001,1010,1011$ &2& 0.003\\
			\hline\hline
		\end{tabular}
		\label{table:2a}
	\end{table}
	
	Similarly when $n=1$ and $m$ varying, the escape rate into the rectangles $R_{i,j,m,1}$ when $i$ is fixed is the same. Further, for fixed $j$, say $j=0$, the number of possible values for the escape rates into the collection of holes $\{R_{i,0,m,1}\ \vert\ 0\leq i\leq M^m-1\}$ is same as the number of possible values for the escape rates into the collection of holes $\{I_{i,m}\ \vert\ 0\leq i\leq M^m-1\}$ for the map $T_M$ on $\mathbb{T}$. Moreover, $\rho(R_{i_1,0,m,1})<\rho(R_{i_2,0,m,1})$ if and only if $\rho_{T_M}(I_{i_1,m})<\rho_{T_M}(I_{i_2,m})$. The case of $m=2$ is described here, refer to Figure~\ref{fig:3}.
	
	Consider Table~\ref{table:2a} where $m=4$ and $n=2$. In this case we have each rectangle corresponds to $2^{4-2}=4$ words of length $4$.
	Note that $R_{0,1,4,2},R_{10,0,4,2}$ and $R_{0,2,4,2}$ have same minimal period $\tau_{\text{min}}$ but different escape rates. Hence, it is clear that unlike one-dimensional case, escape rate does not depend only on the minimal period.
\end{rem}	

\section{Comparing the escape rates - A basic rectangle versus union of basic rectangles with same measure}\label{sec:compare}
Let $T$ be a product of expansive Markov maps on $I^k$ which is conjugate to a full shift on $q$ symbols and $\tilde{\mu}$ be the Markov measure on $\Sigma_{q}^+$. The results in this section hold true for any map which is conjugate to a full shift. 
%Let $T_{M,N}:\mathbb{T}^2\to\mathbb{T}^2$ be the map as defined in Section~\ref{sec:torus_map}.
In this section, we will compare the escape rate into a hole in $I^k$ that corresponds to one cylinder with another hole in $I^k$ that corresponds to a union of cylinders (with certain conditions on the correlation of words), both of which have the same measure. 

Let $\Lambda=\{0,1,\dots,q-1\}$. Given any word $w$ of length $n$ with symbols from $\Lambda$, the measure of the corresponding rectangle $R_w$ in $I^k$ is $\tilde{\mu}(R_w)=\frac{1}{q^n}$ (see Remark~\ref{rem:meas}, 3)). Let $m\geq n$ be given. We consider the collection of words $w_1,w_2,\dots,w_k$ with symbols from $\Lambda$ each of length $\vert w_i\vert=m$ such that the union of rectangles $R_{w_1}\cup R_{w_2}\cup\dots\cup R_{w_k}$ has the same measure as $R_w$. \\
Further $\tilde{\mu}(R_w)=\tilde{\mu}(R_{w_1}\cup R_{w_2}\cup\dots\cup R_{w_k})$ implies $\frac{k}{q^m}=\frac{1}{q^n}$, and hence $k=q^{m-n}$. Moreover, we will consider rectangles $R_{w_1},R_{w_2},\dots,R_{w_k}$ satisfying the following properties;\\~\\
\textbf{(P)} $(w_iw_i)_z=z^{m-1}$ and $(w_iw_j)_z=0$, for every $i\neq j$. \\~\\
Construction of such words with property (P) will be described in Section~\ref{sec:construction}. For the remainder of this section, we assume that a collection of $k$ words $w_1,\dots,w_k$ with property (P) exists. The next result gives the relationship between escape rates into $R_w$ and the union $R_{w_1}\cup R_{w_2}\cup\dots\cup R_{w_k}$. It will be proved later in Remark~\ref{rem:proof}. 

\begin{thm}\label{thm:comparison}
	Let $w$ be any word of length $n$ and $w_1,w_2,\dots,w_k$ be a collection of words of length $m\geq n$ which satisfies the property (P), all with symbols from $\Lambda$, with $k=q^{m-n}$. Then 
	\[
	\rho(R_w)<\rho(R_{w_1}\cup R_{w_2}\cup\dots\cup R_{w_k}).
	\]
\end{thm}

\subsection{Escape rate computation} Let $F(z)$ denote the generating function corresponding to the collection of $k=q^{m-n}$ words of length $m$ with property (P). From Theorem~\ref{thm:formF}, we have 
\[
F(z)=\frac{z}{(z-q)+a(z)},
\] where $a(z)$ is the sum of entries of $M^{-1}$. Here $M=z^{m-1}\text{Id}_{k\times k}$, hence 
\[
F(z)=\frac{z}{(z-q)+\frac{k}{z^{m-1}}}.
\]
Similarly for a single word $w$ of length $n$, the generating function $G(z)$ is given by
\[
G(z)=\frac{z(ww)_z}{(z-q)(ww)_z+1}.
\] 
Let $F(z)=\sum_{i=0}^{\infty}f_i z^{-i}$, hence
\[
\frac{z^{m}}{(z-q)z^{m-1}+k}=\sum_{i=0}^{\infty}f_i z^{-i}.
\]
Comparing the coefficients of $z^i$ for each $i\in\mathbb{Z}$, we get the recursive relation
\[
f_{m+h}=qf_{m+h-1}-kf_h, \ h\geq 0,
\]
with $f_0=1,f_1=q,f_2=q^2,\dots, f_{m-1}=q^{m-1}$.  
\\
Similarly let $G(z)=\sum_{i=0}^{\infty}g_i z^{-i}$ and $(ww)_z=z^{n-1}+\alpha_2 z^{n-2}+\dots+\alpha_n$ where $\alpha_1,\alpha_2,\dots,\alpha_n$ are either $0$ or $1$. Then equating
\[
\frac{z(ww)_z}{(z-q)(ww)_z+1}=\sum_{i=0}^{\infty}g_i z^{-i},\]
and comparing the coefficients, we get the following recursive relation
\begin{eqnarray*}
	g_{n+h}&=& (q-\alpha_2)g_{n+h-1}+(q\alpha_2-\alpha_3)g_{n+h-2}+\dots+(q\alpha_{n-1}-\alpha_n)g_{h+1}\\
	& & +(q\alpha_n-1)g_h, \ h\geq 0,
\end{eqnarray*}
with $g_0=1,g_1=q,g_2=q^2,\dots,g_{n-1}=q^{n-1}$.\\

\noindent Let us assume $(ww)_z=z^{n-1}$. That is, the minimal period of any sequence in $C_w$ is $n$, then we get the following recursive relation, 
\[
g_{n+h}=qg_{n+h-1}-g_h,\ h\geq 0,
\] 
with $g_0=1,g_1=q,g_2=q^2,\dots,g_{n-1}=q^{n-1}$.\\
Let $f_h=r^h$, then the characteristic polynomial for the generating function $F(z)$ becomes $r^m-qr^{m-1}+k=0$. Similarly for $G(z)$, characteristic polynomial will be $r^n-qr^{n-1}+1=0$. Hence generally, for $m\geq n$, the characteristic polynomial for the generating function is given by 
\[
p_{m,n}(r)=r^m-qr^{m-1}+k,
\] 
where $k=q^{m-n}$.\\
If $\mu_1,\mu_2,\dots,\mu_m$ are the $m$ roots of $p_{m,n}$ (counting multiplicity), then for $h\geq 0$, 
\[
f_h=c_1\mu_1^h+c_2\mu_2^h+\dots+c_m\mu_m^h,
\] 
for some appropriate constants $c_1,\dots,c_m$ which are computed using the initial conditions in the recurrence relation.

\begin{thm}\label{thm:large}
	For $n\geq 2$ and $m< (n-1)+\dfrac{(n-1)\ln(q-1)}{\ln(q)-\ln(q-1)}$, the polynomial $p_{m,n}$ has a simple positive real root $\mu_{m,n}$ and all the other roots have modulus less than $\mu_{m,n}$.
\end{thm}

\begin{proof}
	Set $M_n=n+\dfrac{(n-1)\ln(q-1)}{\ln(q)-\ln(q-1)}$. Using Descartes' rules of sign, we can count the number of positive real zeros of $p_{m,n}$. It is easy to see that $p_{m,n}$ has either none or two positive real roots.
	Note that for $m<M_n$,
	\[
	p_{m,n}(0)>0, \ p_{m,n}(q-1)<0, \ p_{m,n}(q)>0,
	\]
	hence $p_{m,n}$ has two positive real roots, one in the interval $(0,q-1)$ and other in the interval $(q-1,q)$. Let $\mu_{m,n}$ be the unique positive real root lying in the interval $(q-1,q)$. We claim that all the other roots of $p_{m,n}$ have modulus less than $\mu_{m,n}$.\\
	Let $m<M_n-1$. From previous arguments, $\mu_{m+1,n}$ exists and lies in the interval $(q-1,q)$ since $m+1<M_n$. Consider 
	\begin{eqnarray*}
		p_{m,n}(\mu_{m+1,n})&=& \mu_{m+1,n}^m-q\mu_{m+1,n}^{m-1}+k\\
		&=&\frac{ \mu_{m+1,n}^{m+1}-q\mu_{m+1,n}^{m}+k\mu_{m+1,n}}{\mu_{m+1,n}}<\frac{p_{m+1,n}(\mu_{m+1,n})}{\mu_{m+1,n}}=0.
	\end{eqnarray*}
	Hence $\mu_{m,n}\in(\mu_{m+1,n},q).$
	%	We first note that for all $m< M_n-1$, we have $\mu_{m+1,n}<\mu_{m,n}$ since $p_{m,n}(\mu_{m+1,n})<0$ (as $m+1<M_n$, $\mu_{m+1,n}\in (q-1,q)$ exists).
	Also $\mu_{m,n}\neq \mu_{m+1,n}$, since if they were equal then $p_{m+1,n}(\mu_{m,n})=p_{m,n}(\mu_{m,n})=0$ implies $\mu_{m,n}=1<q-1$, which is not true.\\
	We will now prove the claim stated before. For that, we will show that there exists exactly $m-1$ roots of $p_{m,n}$ (counting multiplicity) inside the ball of radius $\mu_{m+1,n}$, which implies that $\mu_{m,n}$ is the positive real root of $p_{m,n}$ with largest modulus.\\
	By Rouche's theorem, this happens if $q(\mu_{m+1,n})^{m-1}>(\mu_{m+1,n})^{m}+k$. Since $\mu_{m+1,n}$ is a root of $p_{m+1,n}$, we have $(\mu_{m+1,n})^{m+1}-q(\mu_{m+1,n})^{m}+qk=0$. Note that since $\mu_{m+1,n}<q$, we have 
	\begin{eqnarray*}
		(\mu_{m+1,n})^{m-1} (q-\mu_{m+1,n})^2>0 \\
		\implies q^2(\mu_{m+1,n})^{m-1}-q(\mu_{m+1,n})^m-qk>0\\
		\implies q(\mu_{m+1,n})^{m-1}-(\mu_{m+1,n})^m-k>0 
	\end{eqnarray*}
	which proves the result. 
\end{proof}

\begin{rem} \label{rem:rho_mn}
	From Theorem~\ref{thm:large}, $p_{m,n}$ has a positive real root $\mu_1=\mu_{m,n}$ such that all the other roots of $p_{m,n}$ have modulus less than $\mu_{1}$. Let $\rho_{m,n}$ denote the escape rate into the hole that corresponds to the union of cylinders based at $k=q^{m-n}$ words of length $m$ with property P. Then 
	\begin{eqnarray*}
		\rho_{m,n}&=&-\lim_{h \to \infty}\frac{1}{h}\ln 
		\frac{f_{h+m}}{q^{h+m}}\\
		&=&\ln q-\lim_{h \to \infty}\frac{1}{h}\ln\mu_1^{h+m}\left(\dfrac{c_1\mu_1^{h+m}+c_2\mu_2^{h+m}+\dots+c_m\mu_m^{h+m}}{\mu_1^{h+m}}\right)\\ &=& -\ln\left(\frac{\mu_1}{q}\right)>0. 
	\end{eqnarray*}
\end{rem}
\noindent Next we see that for a fixed $n$, as $m$ increases, the escape rate increases, the proof of which is immediate from Remark~\ref{rem:rho_mn}, and the fact that $\mu_{m+1,n}<\mu_{m,n}$ from Theorem~\ref{thm:large}.
\begin{thm}\label{thm:large1}
	Fix $n\geq 1$ and let $n\leq m<(n-1)+\dfrac{(n-1)\ln(q-1)}{\ln(q)-\ln(q-1)}$. Then
	\[
	\rho_{m,n}<\rho_{m+1,n}.
	\]
\end{thm}

\begin{rem}\label{rem:proof}
	1)  In Theorem~\ref{thm:large1}, if $m=n$, the hole is a rectangle that corresponds to a cylinder in which the minimal period of any sequence is $n$. The escape rate increases as the number of rectangles in the union collection increases. Moreover as the number of rectangles increases, the measure of each rectangle decreases.\\
	2) In~\cite[Lemma 4.5.1]{BY}, it was shown that if $u_1,u_2$ are two words of the same length, then $(u_1u_1)_q>(u_2u_2)_q$ (which is the autocorrelation polynomial evaluated at $z=q$) implies $\rho(R_{u_2})>\rho(R_{u_1})$. Thus, let $u$ be any word of length $n$ (with arbitrary autocorrelation) and for $m\geq n$, $u_1,u_2,\dots,u_m$ be a collection of $m$ words with property (P), then 
	\[
	\rho(R_u)<\rho(R_w)<\rho(R_{u_1}\cup R_{u_2}\cup\dots\cup R_{u_m}),
	\]
	where $w$ is any word of length $n$ with $(ww)_z=z^{n-1}$. This proves Theorem~\ref{thm:comparison}.
\end{rem}

\subsection{Construction of words with property (P)} \label{sec:construction}
Let $\Lambda=\{0,1,\dots ,q-1\}$ as before. We will construct $k=q^{m-n}$ words of length $m$ with symbols from $\Lambda$ which have property (P), that is, each of which have autocorrelation polynomial $z^{m-1}$ and whose cross-correlation polynomials are zero. 

\subsubsection{Construction 1}
Consider the collection
\[
S= \left\{ w(q-1)\ \vert \ w\in \{0,1,\dots,q-2\}^{m-1}\right\}.
\]
Note that $\#S =(q-1)^{m-1}$. The union of cylinders based at words in the collection $S$ correspond to the union $\bigcup_{v\in S}R_v$ of $(q-1)^{m-1}$ rectangles each of measure $\dfrac{1}{q^m}$. \\
For the map $T_{3,2}$ (as defined in~\eqref{eq:map}), these rectangles are scattered on the square $[0,1]\times [0,1]$ as shown in Figure~\ref{fig:6}, for $m=3$ and $q=6$. In Figure~\ref{fig:6}, the shaded rectangles correspond to the collection $S$. The collection $S$ is given by
\[
S= \left\{ ab5\ \vert \ a,b\in \{0,1,\dots,4\}\right\}.
\] 
It is clear that the number of shaded rectangles is 25.

\begin{figure}[h!]
	\centering
	\includegraphics[width=.8\textwidth]{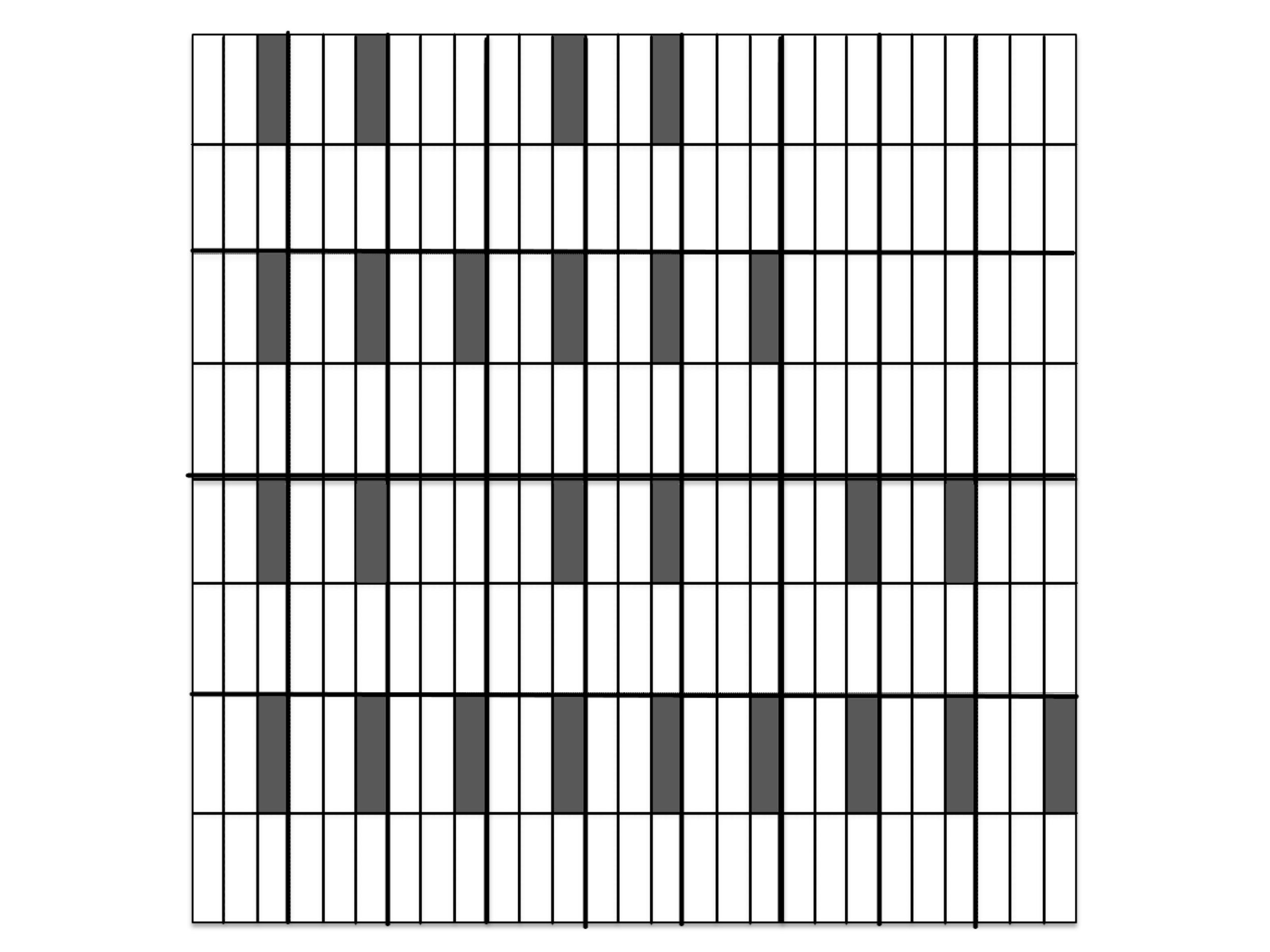}
	\caption{Rectangles corresponding to the collection $S$ in Construction 1 for $m=3$ and $q=6$ for the map $T_{3,2}$ on $\mathbb{T}^2$.}
	\label{fig:6}
\end{figure}

\noindent Note that 
\begin{eqnarray*}
	\#S\geq k \iff (q-1)^{m-1}\geq q^{m-n}\Rightarrow m &\leq& n+\dfrac{(n-1)\ln(q-1)}{\ln(q)-\ln(q-1)},
\end{eqnarray*} 
which appeared in Theorem~\ref{thm:large}. \\
Let $u=u_1u_2\dots u_m$ be any word of length $m$, which is not in $S$. Then either $u_m\neq q-1$ or there exists $p\in\{1,2,\dots,m-1\}$ such that $u_p=q-1$. The first situation cannot happen since otherwise the cross-correlation of $u$ with the word $v=u_m0\dots0(q-1)\in S$, i.e., $(uv)_z$ is non-zero. Now let $p$ be the first position where $u_p=q-1$. Then $v=00\dots 0u_1u_2\dots u_p\in S$ and $(vu)_z$ is non-zero. Hence this collection $S$ is maximal in the sense that if one more word $u$ of length $m$ with $(uu)_z=z^{m-1}$ is added to $S$, then there exists a word in $S$ whose cross-correlation with $u$ is non-zero.

\subsubsection{Construction 2} 
We can generalize Construction 1 as follows:\\
Choose $1\leq \ell< q-1$, and $\ell$ distinct symbols, say $i_1,i_2,\dots,i_\ell$, from $\Lambda$. \\
Consider the collection
\[
S=\{wi_r \vert \ w \text{ is a word that does not contain } i_1,i_2,\dots,i_\ell,1\leq r\leq \ell\}.
\]
Note that the first construction described is when $\ell=1$ and $i_1=q-1$ (there was nothing special about $q-1$ in Construction 1, we could have chosen any other symbol). Also 
$\#S=\ell(q-\ell)^{m-1}$. Moreover $S$ corresponds to the union $\bigcup_{v\in S}R_v$ of $\ell(q-\ell)^{m-1}$ rectangles each of measure $\dfrac{1}{q^m}$. For $S$ to have at least $k$ elements, $m$ should satisfy the following upper bound
\[ m\leq n+\dfrac{(n-1)\ln(q-\ell)+\ln(\ell)}{\ln(q)-\ln(q-\ell)}.
\]

\begin{rem}
	1) When $m>q$, the number of words in Construction 2 (that is, $l(q-\ell)^{m-\ell}$) is less than the number of words in Construction 1, which is $(q-1)^{m-1}$.\\
	2) For $n\geq 2$, it can be proved that the upper bound for $m$ obtained (as a function of $\ell$) is maximum when $\ell=1$. 
\end{rem}

\subsubsection{Construction 3} 
Choose $1\leq \ell< q-1$ and $\ell$ different symbols from $\Lambda$, say $i_1,i_2,\dots,i_\ell$. Let $1\leq r<m$. \\
Consider the following set of words 
\begin{eqnarray*}
	S &=& \{wu \vert \ w \text{ is a word of length $m-r$ that does not contain } \\
	& & i_1,i_2,\dots,i_\ell, \ \text{and } u\in \{i_1,\dots,i_\ell\}^{r}\}.
\end{eqnarray*}
Note that $\#S =(q-\ell)^{m-r}\ell^{r}$. Moreover $S$ corresponds to the union $\bigcup_{v\in S}R_v$ of $(q-\ell)^{m-r}\ell^{r}$ rectangles each of measure $\dfrac{1}{q^m}$.  For $S$ to have at least $k$ elements, $m$ should satisfy the following upper bound
\begin{eqnarray*}
	m &\leq& n+\dfrac{(n-r)\ln(q-\ell)+r\ln \ell}{\ln(q)-\ln(q-\ell)}.
\end{eqnarray*} 

\begin{rem}
	Construction 1 is a special case of Construction 3 when $r=1$ and $\ell=1$. Also Construction 2 is a special case of Construction 3 when $r=1$. 
\end{rem}

\begin{table}
	\centering
	\caption{Upper bound on $m$ in Construction 2 with $q=6$.}	
	\begin{tabular}{|c c c c c |} 
		\hline
		$n$ & $\ell=1$ & $\ell=2$& $\ell=3$ & $\ell=4$ \\ [0.5ex] 
		\hline\hline
		1&1&2&2&2\\\hline
		2&10&7&5&3\\\hline
		3&20&11&7&5\\\hline
		4&30&15&10&7\\\hline
		5&40&20&12&8\\\hline
		6&50&24&15&10\\\hline
		7&59&29&18&12\\\hline
		8&69&33&20&13\\\hline
		9&79&38&23&15\\\hline
		\hline
	\end{tabular}
	\label{table:3}
\end{table}

\noindent Taking $q=6$ and fixing $r=1$ in Construction 3 and thus looking at Constructions 1 and 2, we get Table~\ref{table:3} showing the largest value for $m$ obtained for different values of $n$ and $\ell$. Note that Construction 1 ($\ell=1$) gives largest set $S$ except when $n=1$.

\section{Subshift of finite type} \label{sec:subshift}
In this section, we consider the case where the product map is conjugate to the shift map on $\Sigma_{\mathcal{F}}$, where $\mathcal{F}$ is a collection of forbidden words. Let the hole $H$ correspond to another collection of forbidden words $\mathcal{F}_1$, with $\mathcal{F}$ and $\mathcal{F}_1$ consisting of words of same length. Then, by Theorem~\ref{thm:subshift},
\begin{eqnarray*}
	\rho(H)&=& h_{\text{top}}(\Sigma_{\mathcal{F}}) - h_{\text{top}}(\Sigma_{\mathcal{F}\cup \mathcal{F}_1}).
\end{eqnarray*}
We now give a few examples to illustrate this, with notations as before. The escape rates are calculated using the approach described in Section~\ref{sec:prod}.
\begin{exam}
	Let $f_1$ be any map on $I$ conjugate to the shift map on a subshift of finite type $\Sigma_{A_1}$ with transition matrix $A_1=\begin{pmatrix}
	1&1\\
	1&0
	\end{pmatrix}$. Consider the product map $f=f_1\times f_1$. Here $k=2$, $N_1=N_2=2$, and hence $N=N_1N_2=4$ and $\phi(0)=(0,0)$, $\phi(1)=(0,1)$, $\phi(2)=(1,0)$, and $\phi(3)=(1,1)$. Moreover $f$ is conjugate to the shift map on $\Sigma_A$ with 
	\[
	A=A_1\otimes A_1=\begin{pmatrix}
	1&1&1&1\\
	1&0&1&0\\
	1&1&0&0\\
	1&0&0&0
	\end{pmatrix}.
	\]
	Thus $\Sigma_{\mathcal{F}}=\Sigma_A$ consists of all the sequences in $\Sigma_4^+=\{0,1,2,3\}^\mathbb{N}$ with 
	\[
	\mathcal{F}= \{11, 13, 22, 23, 31, 32, 33\}. 
	\]
	The Perron (largest real) eigenvalue of $A$ is $\lambda_{\text{max}}=\frac{3+\sqrt{5}}{2}$ with normalized left and right eigenvectors given by $u=(0.723,0.447,0.447,0.276)$ and $v=u^T$, respectively. Note that $h_{\text{top}}(\Sigma_{\mathcal{F}})=\ln(\lambda_{\text{max}})=0.962$. 
	
	\begin{table}[h!]
		\centering
		\caption{Escape rates for $f$ into holes corresponding to a cylinder based at an allowed word of length two.}
		\begin{tabular}{|c|c|c|c|}
			\hline\hline
			Holes $H=R_{ij}$& $\tilde{\mu}(H)$ & $\rho(H)\sim$ & $\tau_{\text{min}}(H)$  \\\hline
			$R_{00}$ & 0.2 & 0.188 & 1\\\hline
			$R_{01}$, $R_{02}$, $R_{10}$, $R_{20}$ & 0.124 & 0.153 & 2\\\hline
			$R_{03}$, $R_{12}$, $R_{21}$, $R_{30}$ & 0.076 &0.081 & 2\\ \hline\hline	
		\end{tabular}\label{table:subshift1}
	\end{table}
	
	\begin{table}
		\centering
		\caption{Escape rates for $f$ into holes corresponding to a cylinder based at an allowed word of length three.}
		\begin{tabular}{|c|c|c|c|}
			\hline\hline
			Holes $H=R_{ijk}$& $\tilde{\mu}(H)$ & $\rho(H)\sim$ & $\tau_{\text{min}}(H)$  \\\hline
			$R_{010}$, $R_{020}$, $R_{030}$ & 0.076 & 0.081 & 2\\\hline
			$R_{000}$ & 0.076 & 0.057 & 1\\\hline
			$R_{001}$, $R_{002}$, $R_{012}$, $R_{021}$, $R_{100}$, $R_{120}$, $R_{200}$, $R_{210}$ & 0.047 &0.054 & 3\\ \hline
			$R_{003}$, $R_{102}$, $R_{201}$, $R_{300}$& 0.029 & 0.031 & 3\\\hline
			$R_{101}$, $R_{121}$, $R_{202}$, $R_{212}$& 0.029 & 0.028 & 2\\\hline
			$R_{103}$, $R_{203}$, $R_{301}$, $R_{302}$& 0.018 & 0.019 & 3\\\hline
			$R_{303}$& 0.011 & 0.010 & 2\\
			\hline\hline	
		\end{tabular} \label{table:subshift2}
	\end{table}
	
	\noindent Consider the holes $R_{00}$, $R_{01}$, $R_{02}$, $R_{03}$, $R_{10}$, $R_{12}$, $R_{20}$, $R_{21}$, $R_{30}$ corresponding to cylinders based at words of length two. Thus $\mathcal{F}_1=\{ij\}$. As discussed before, these holes need not have equal (Parry) measure even though all of these correspond to words of length two. The set $I^2\setminus\Omega_k(R_{ij})$ consists of all points in $I^2$ which correspond to sequences in $\Sigma_{\mathcal{F}}$ that do not contain the word $ij$ in their first $k+2$ positions. We did not consider holes corresponding to words in $\mathcal{F}$ since they are anyway forbidden in sequences in $\Sigma_{\mathcal{F}}$. See Table~\ref{table:subshift1} for the measure $\tilde{\mu}(H)$ of corresponding holes and escape rate $\rho(H)$ into them.
	
	Now consider the holes $R_{000}$, $R_{001}$, $R_{002}$, $R_{003}$, $R_{010}$, $R_{012}$, $R_{020}$, $R_{021}$, $R_{030}$, $R_{100}$, $R_{101}$, $R_{102}$, $R_{103}$, $R_{120}$, $R_{121}$, $R_{200}$, $R_{201}$, $R_{202}$, $R_{203}$, $R_{210}$, $R_{212}$, $R_{300}$, $R_{301}$, $R_{302}$, $R_{303}$ corresponding to cylinders based at words of length three. Thus $\mathcal{F}_1=\{ijk\}$. The set $I^2\setminus\Omega_k(R_{ijk})$ consists of all points in $I^2$ which correspond to sequences in $\Sigma_{\mathcal{F}}$ that do not contain the word $ijk$ in their first $k+3$ positions. Again note that we did not consider holes corresponding to words in $\mathcal{F}$. See Table~\ref{table:subshift2} for the measure $\tilde{\mu}(H)$ of corresponding holes and escape rate $\rho(H)$ into them.
	
	Note that in this example, $R_{010}$ and $R_{000}$ have equal measure, but different escape rate which illustrates that escape rate depends on the position of the hole (or, the correlation of the corresponding words), as was observed in~\cite{BY} for the full shift case.
\end{exam}

\begin{exam}
	Let $T_2:I\to I$ be defined as $T_2(x)=2x\ (\text{mod}\ 1)$ with the transition matrix $A_1=\begin{pmatrix}
	1&1\\1&1
	\end{pmatrix}$. Let $S$ be any map on $I$ conjugate to the shift map on a subshift of finite type with transition matrix $A_2=\begin{pmatrix}
	1&1\\1&0
	\end{pmatrix}$. Consider the product map $f=T_2\times T_2\times S:I^3\to I^3$. Here $N_1=N_2=N_3=2$, hence $N=8$ and $\phi(0)=(0,0,0),\phi(1)=(0,0,1),\phi(2)=(0,1,0),\phi(3)=(0,1,1),\phi(4)=(1,0,0),\phi(5)=(1,0,1),\phi(6)=(1,1,0),$ and $\phi(7)=(1,1,1)$. The transition matrix for $f$ is given by
	\[
	A=A_1\otimes A_1\otimes A_2=\begin{pmatrix}
	1&1&1&1&1&1&1&1\\
	1&0&1&0&1&0&1&0\\
	1&1&1&1&1&1&1&1\\
	1&0&1&0&1&0&1&0\\
	1&1&1&1&1&1&1&1\\
	1&0&1&0&1&0&1&0\\
	1&1&1&1&1&1&1&1\\
	1&0&1&0&1&0&1&0
	\end{pmatrix}.
	\]
	
	\noindent Note that the collection of forbidden words is 
	\[
	\mathcal{F}=\{ab: a,b \text{  are odd numbers, }0\leq a,b\leq 7\}.
	\]
	\begin{table}
		\centering
		\caption{Escape rates for $f=T_2\times T_2\times S$ into holes corresponding to a cylinder based at an allowed word of length two.}
		\begin{tabular}{|c|c|c|c|}
			\hline\hline
			Holes $H$& $\tilde{\mu}(H)\sim $ & $\rho(H)\sim$ & $\tau_{\text{min}}(H)$  \\\hline
			$R_{aa}$, $a$ is even & 0.0279  & 0.0251 & 1\\\hline
			$R_{ab}$, exactly one of $a$ or $b$ is even  & 0.0173 & 0.0176 & 2\\\hline
			$R_{ab}$, both $a\ne b$ are even & 0.0279 &0.0293 & 2\\ \hline\hline	
		\end{tabular}\label{table:subshift3}
	\end{table}
	
	\noindent Table~\ref{table:subshift3} shows the escape rates into holes corresponding to cylinders based at words of length two. Here, $\tau_{\text{min}}(R_{00})<\tau_{\text{min}}(R_{01})$ but $\rho(R_{00})>\rho(R_{01})$, and $\tau_{\text{min}}(R_{00})<\tau_{\text{min}}(R_{02})$ but $\rho(R_{00})<\rho(R_{02})$. Thus, unlike the one-dimensional case, escape rate into the holes for the product of maps does not always depend on the minimal period of the hole (see also Theorem~\ref{thm:tau_rho}). 
\end{exam}

\begin{exam}\label{exam_123}
	Even if the holes corresponding to the collections $\mathcal{F}_1$ and $\mathcal{F}_2$ have similar words in the sense that they give same correlation polynomials, their cross-correlations with words from $\mathcal{F}$ may be different. Hence, by Theorem~\ref{thm:subshift}, they may have different escape rates. It can be observed in the one-dimensional map case itself. 
	\begin{figure}[h!]
		\centering
		\begin{subfigure}{.45\textwidth}
			\centering
			\includegraphics[width=1.0\linewidth]{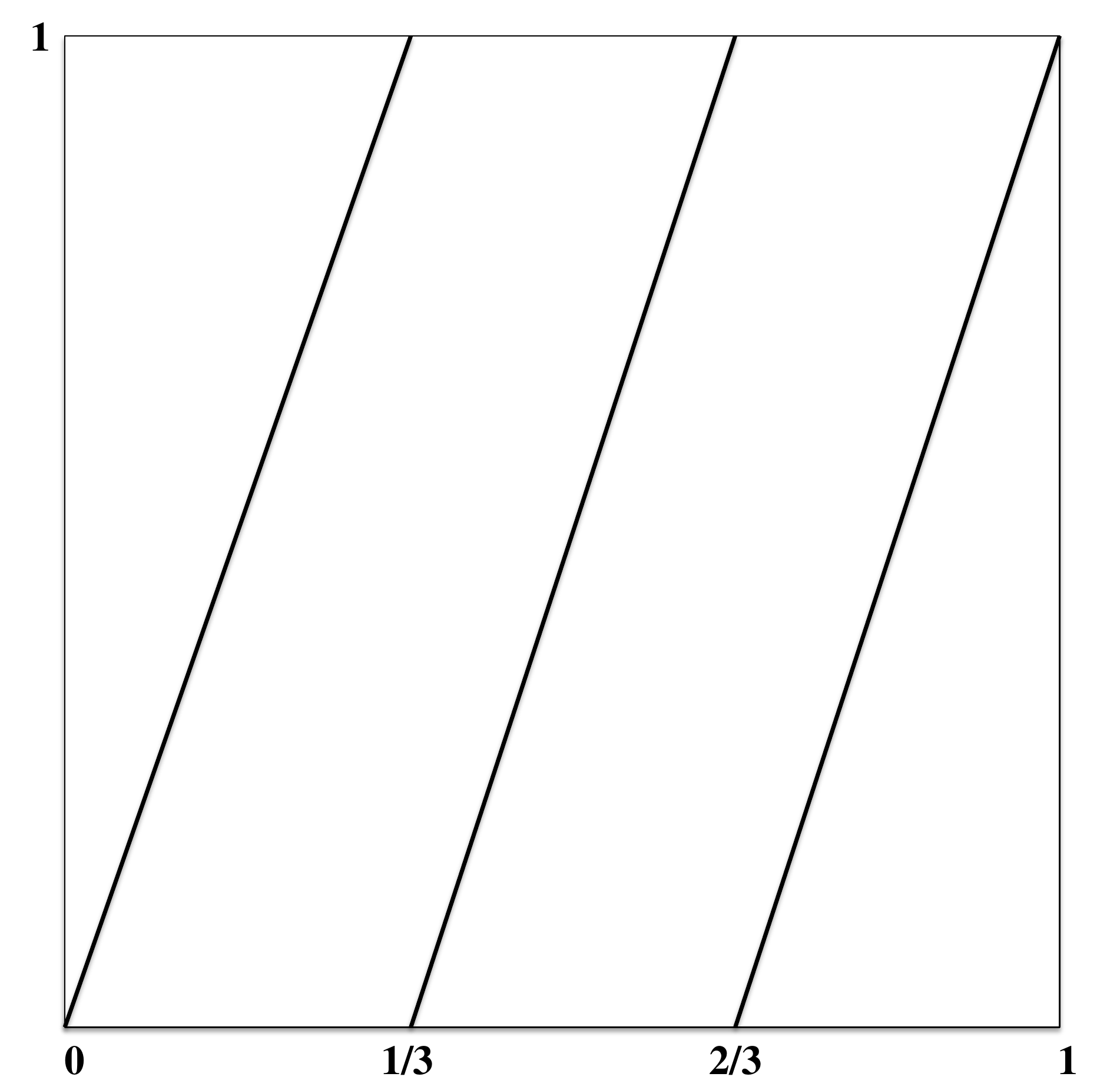}
			\caption*{Map $T_1$}
		\end{subfigure}
		\begin{subfigure}{.45\textwidth}
			\centering
			\includegraphics[width=1.0\linewidth]{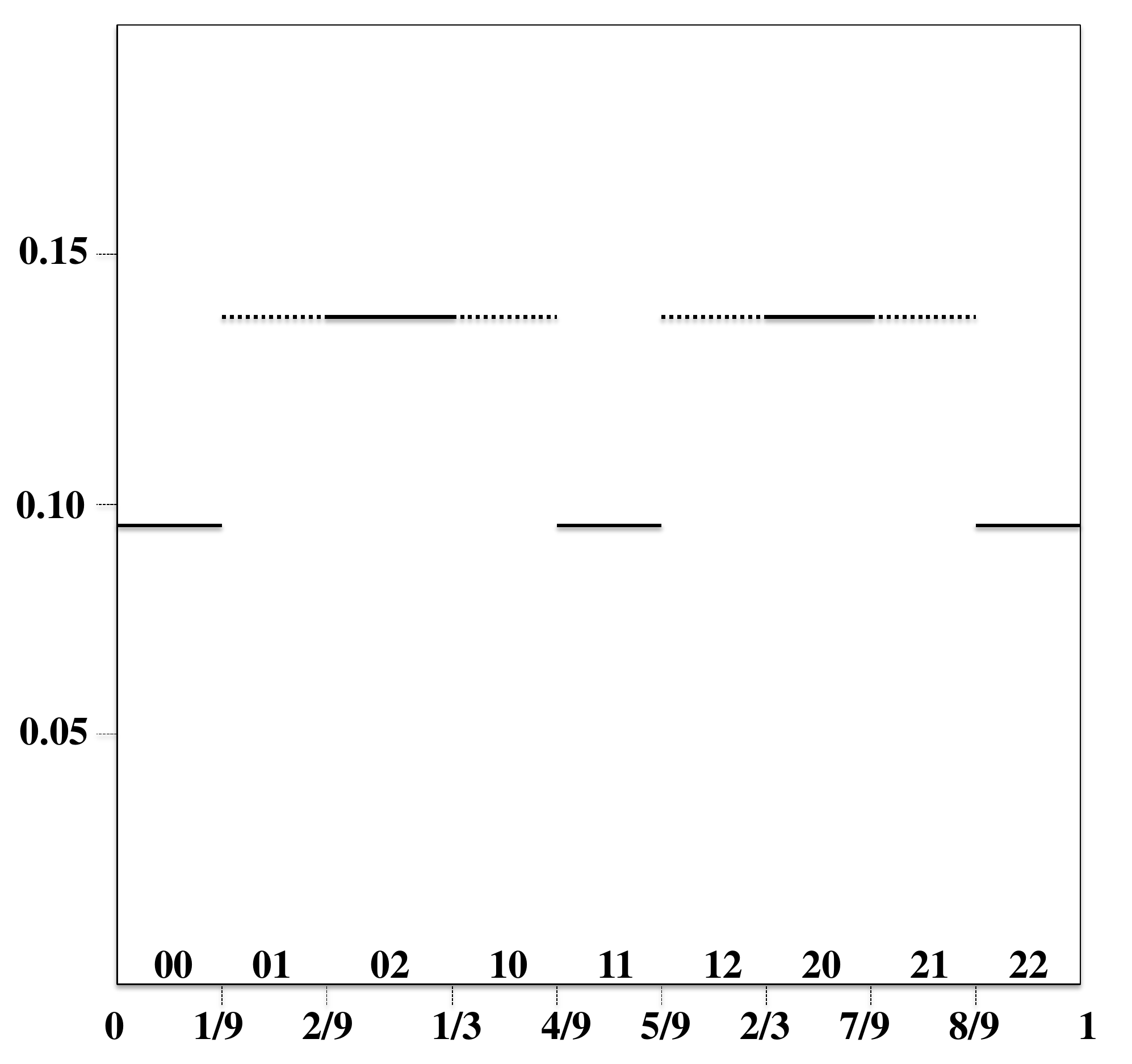}
			\caption*{Escape rates for $T_1$}
		\end{subfigure}
		\begin{subfigure}{.45\textwidth}
			\centering
			\includegraphics[width=1.0\linewidth]{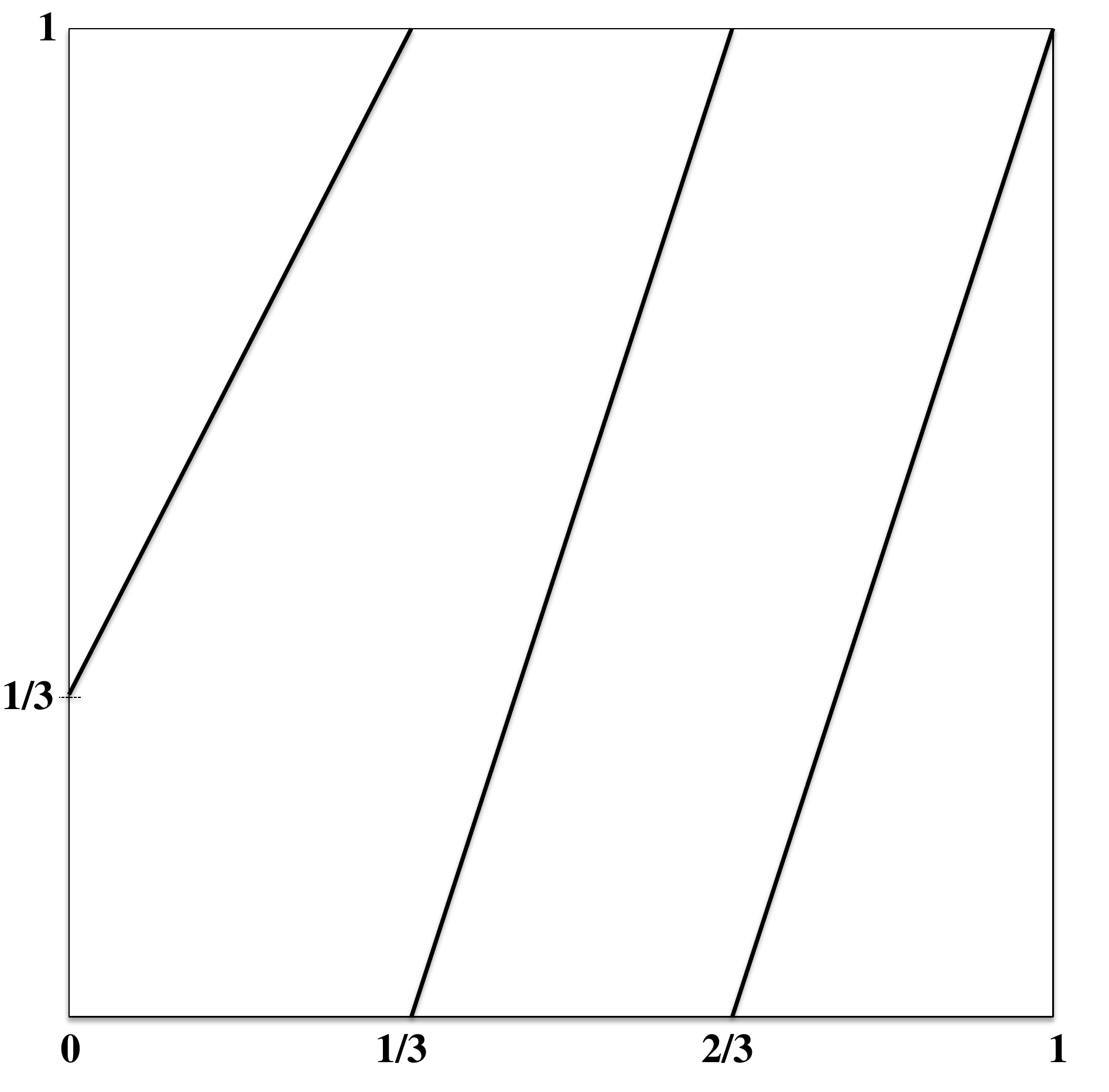}
			\caption*{Map $T_2$}
		\end{subfigure}
		\begin{subfigure}{.45\textwidth}
			\centering
			\includegraphics[width=1.0\linewidth]{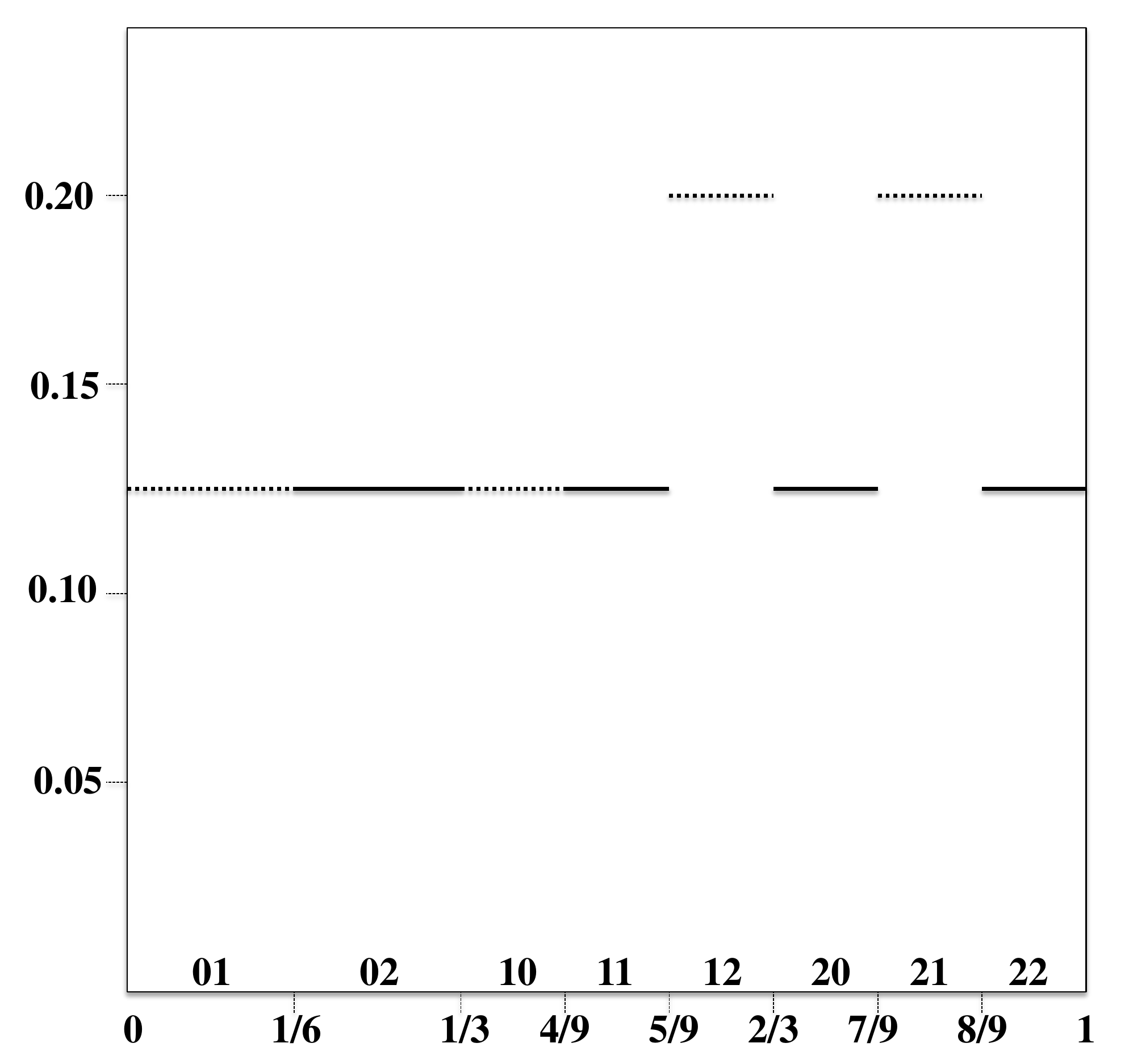}
			\caption*{Escape rates for $T_2$}
		\end{subfigure}
		\begin{subfigure}{.45\textwidth}
			\centering
			\includegraphics[width=1.0\linewidth]{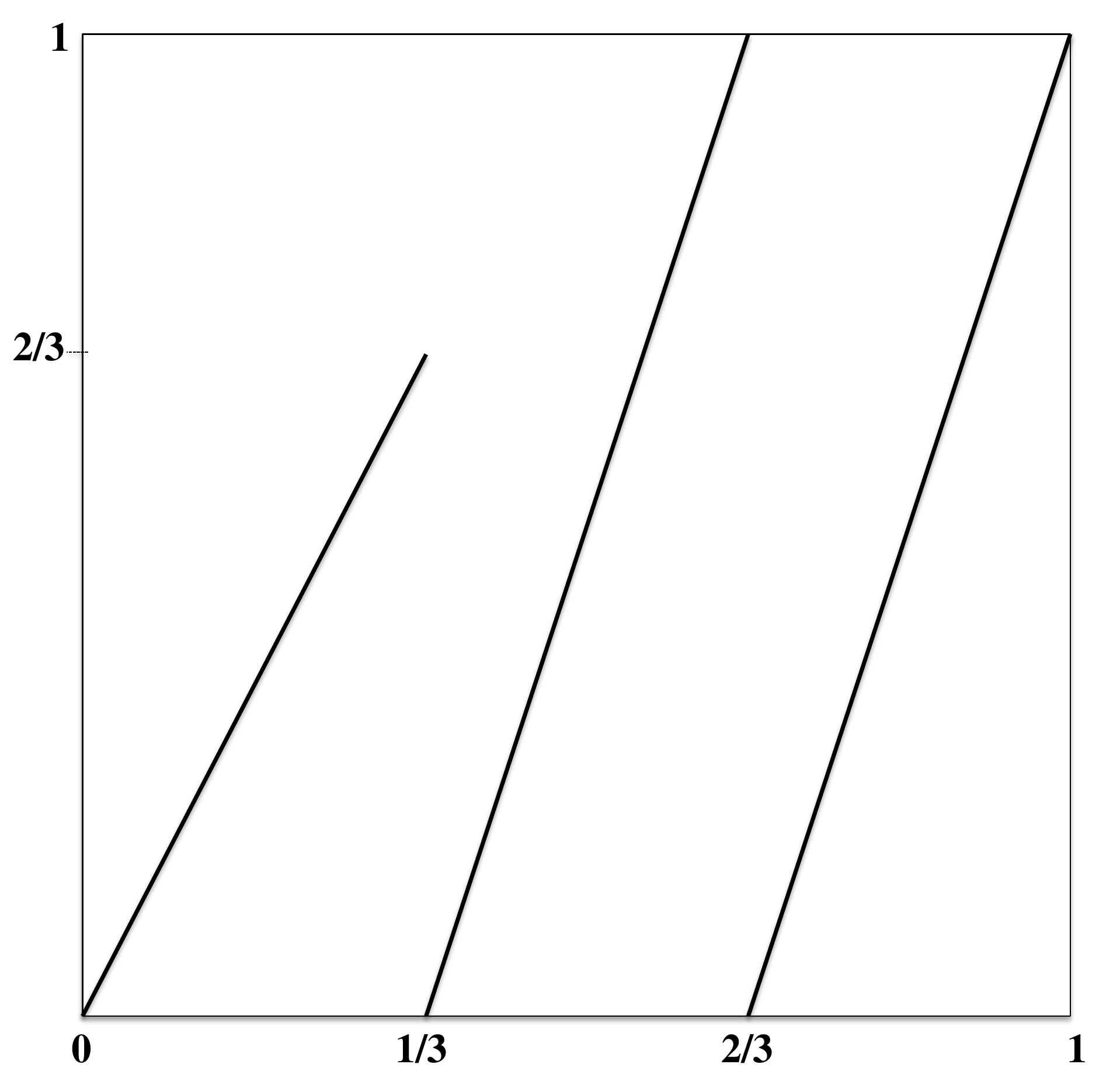}
			\caption*{Map $T_3$}
		\end{subfigure}
		\begin{subfigure}{.45\textwidth}
			\centering
			\includegraphics[width=1.0\linewidth]{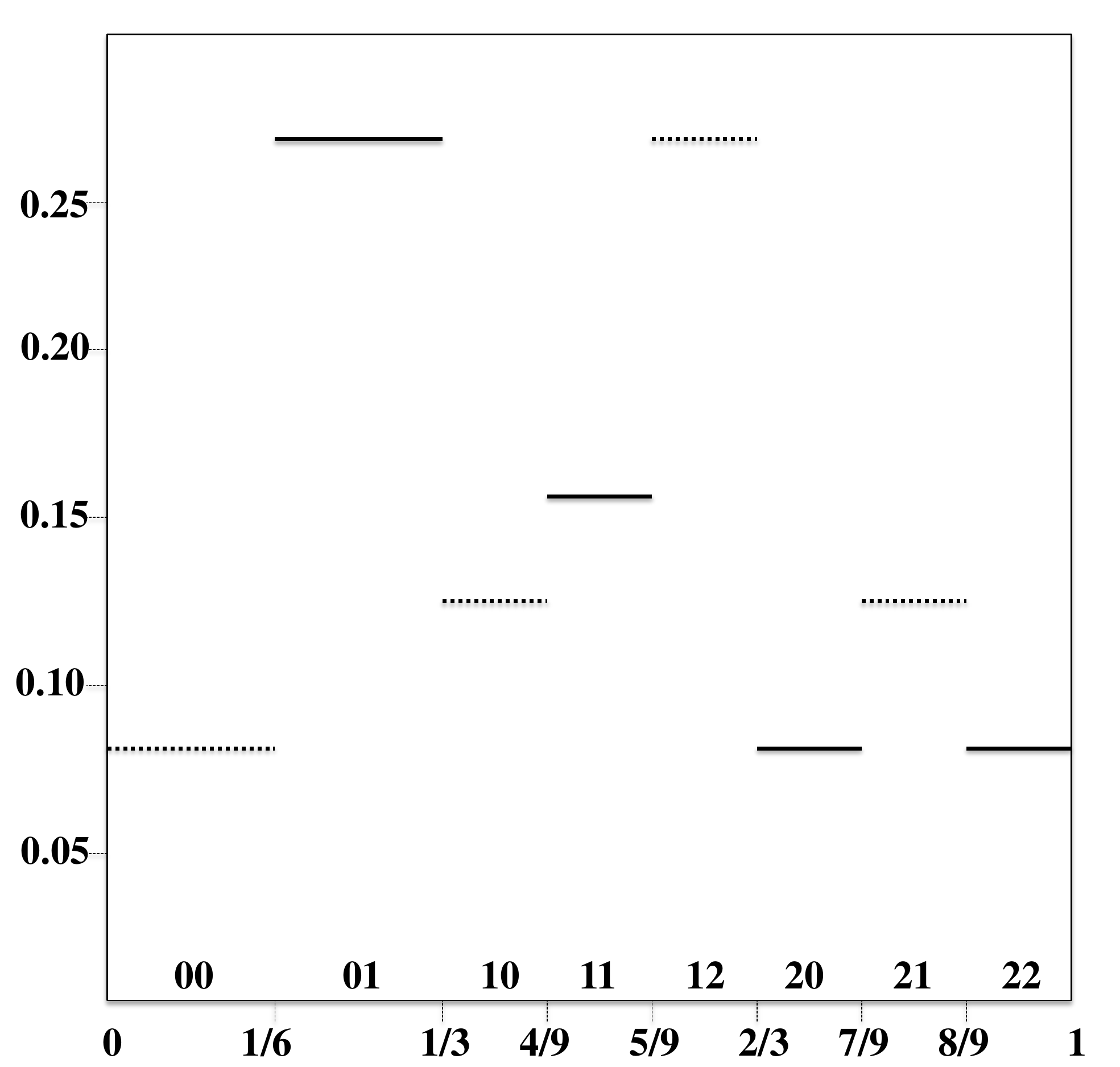}
			\caption*{Escape rates for $T_3$}
		\end{subfigure}
		\caption{Maps conjugate to shift map on subshift of finite type.}
		\label{fig:mapT}
	\end{figure}
	
	In Figure~\ref{fig:mapT}, three maps $T_1,T_2$ and $T_3$ on $I=[0,1]$ are shown. All the maps are expansive Markov and hence conjugate to the shift map on some subshift of finite type. Let $I_{w}$ denote the interval contained in $I$ corresponding to the cylinder based at the word $w$.
	
	The map $T_1$ is conjugate to the shift map on the full shift space $\Sigma_3^+$. Hence $\mathcal{F}=\phi$ in this case. We consider the hole corresponding to a cylinder based at a word of length two, i.e. $\mathcal{F}_1=\{ab\}$, $a,b\in\{0,1,2\}$. Since $T_1$ is conjugate to the shift map on full shift space, each of this hole $I_{ab}$ will have equal measure $\tilde{\mu}(I_{ab})=\frac{1}{3^2}$. Note that, in this case, these holes give two possible values of escape rates. Here holes corresponding to cylinders based at words with same autocorrelation polynomial give the same escape rate. Such maps which are conjugate to a full shift were extensively studied in~\cite{BY}.   
	
	Now consider the map $T_2$. The map is conjugate to shift map on $\Sigma_{\mathcal{F}}\subseteq \Sigma_3^+$, where $\mathcal{F}=\{00\}$. In this case, the holes of the type $I_{ab}$, corresponding to cylinder based at an allowed word of length two, do not necessarily have equal measure $\tilde{\mu}$. Also, the escape rate depends on the cross-correlation of words $ab$ and $00$. Hence it exhibits a different pattern for the escape rates. Note that there are two possible escape rates into the holes of the type $I_{ab}$ in this case as well. The collections $\mathcal{F}_1=\{01\}$ and $\mathcal{F}_2=\{12\}$ give same autocorrelation polynomials but they have different cross-correlation polynomials with $\mathcal{F}$, and $\rho(I_{01})\ne\rho(I_{12})$. It is worth noting that $\tilde{\mu}(I_{01})\ne\tilde{\mu}(I_{11})$, but the escape rates are the same. See Table~\ref{table:T_2} for measure $\tilde{\mu}(H)$ of corresponding holes and escape rate $\rho(H)$ into them. Referring to~\eqref{length2}, $a(z)$ is also given in the table for words $00$ and $ab$.  
	
	\begin{table}
		\centering
		\caption{Escape rate for $T_2$ into holes corresponding to a cylinder based at an allowed word of length two.}
		\renewcommand{\arraystretch}{2.5}
		\begin{tabular}{|c|c|c|c|c|}
			\hline\hline
			Hole $H=I_{ab}$& $\tilde{\mu}(H)\sim$ & $a(z)$ & $a(3)$ & $\rho(H)\sim$  \\\hline
			$I_{01},I_{02},I_{10},I_{20}$& 0.1057 & $\dfrac{2}{z+1}$ & $\dfrac{1}{2}$ & 0.1237\\\hline
			$I_{11},I_{22}$ & 0.1443 & $\dfrac{2}{z+1}$ & $\dfrac{1}{2}$ &  0.1237\\\hline
			$I_{12},I_{21}$&0.1443 & $\dfrac{2z+1}{z(z+1)}$ & $\dfrac{7}{12}$ & 0.1955\\\hline\hline
		\end{tabular}\label{table:T_2}
	\end{table}
	
	\begin{table}
		\centering
		\caption{Escape rate for $T_3$ into holes corresponding to a cylinder based at an allowed word of length two.}
		\renewcommand{\arraystretch}{2.5}
		\begin{tabular}{|c|c|c|c|c|}
			\hline\hline
			Hole $H=I_{ab}$& $\tilde{\mu}(H)\sim$ & $a(z)$& $a(3)$ & $\rho(H)\sim$  \\ \hline
			$I_{00},I_{22}$ & 0.1056 & $\dfrac{2}{z+1}$ & $\dfrac{1}{2}$& 0.0810\\\hline
			$I_{01},I_{12}$ & 0.1708 & $\dfrac{2}{z}$ & $\dfrac{2}{3}$& 0.2693\\\hline
			$I_{10},I_{21}$ & 0.1056 & $\dfrac{2z-1}{z^2}$ & $\dfrac{5}{9}$ & 0.1188 \\\hline
			$I_{11}$ & 0.1708 & $\dfrac{2z+1}{z(z+1)}$ & $\dfrac{7}{12}$& 0.1528\\\hline
			$I_{20}$ & 0.0652 & $\dfrac{2}{z+1}$ & $\dfrac{1}{2}$& 0.0810\\\hline
			\hline
		\end{tabular}\label{table:T_3}
	\end{table}
	
	Finally consider the map $T_3$, which is conjugate to the shift map on $\Sigma_{\mathcal{F}}\subseteq \Sigma_3^+$, where $\mathcal{F}=\{02\}$. Similar to the previous case of $T_2$, holes of the type $I_{ab}$ corresponding to cylinders based at words of length two do not have equal measure $\tilde{\mu}$. Unlike for $T_1$ and $T_2$, in this case, there are four possible values of escape rates. Here, the collections $\mathcal{F}_1=\{00\}$ and $\mathcal{F}_2=\{11\}$ give the same autocorrelation polynomials but they have different cross-correlation polynomials with $\mathcal{F}$, and $\rho(I_{00})\ne\rho(I_{11})$. Note that $I_{00}$ and $I_{20}$ have different measures but same escape rates. See Table~\ref{table:T_3} for the measure $\tilde{\mu}(H)$ of corresponding holes, $a(z)$ (for words $02$ and $ab$) and escape rate $\rho(H)$ into them.
	
	In the above examples, $q=3$. The value $a(3)$ is also shown in Tables~\ref{table:T_2} and~\ref{table:T_3}. Recall that if there is only one forbidden word $w$, $a(q)=1/(ww)_q$. For the full shift case, in~\cite[Lemma 4.5.1]{BY}, it was shown that if $w_1,w_2$ are two words of the same length, then $(w_1w_1)_q>(w_2w_2)_q$ implies $\rho(R_2)>\rho(R_1)$, where $R_i$ is the rectangle in $I$ corresponding to the cylinder based at the word $w_i$, $i=1,2$. The maps $T_2$ and $T_3$ exhibit the same pattern. Of course, $R_1$ and $R_2$ have unequal measures. In both the cases, the escape rate increases with the value of $a(3)$. It is worth exploring whether this behaviour is true in general, for maps conjugate to the shift map on some subshift of finite type, when the number of forbidden words is more than one.
\end{exam}

\section{Concluding Remarks}\label{sec:conc}
In this paper, we studied the escape rate into Markov holes for maps which are product of expansive Markov maps. Such maps are conjugate to a subshift of finite type, hence tools from symbolic dynamics can be used. The hole corresponds to a union of cylinders and the escape rate depends on the correlations of the associated words and the size of the set of symbols. We have presented examples of product of expansive Markov maps, specifically product of two $k$-expanding transformations, in which case, their product is conjugate to a full shift.

We have presented several examples to illustrate the general situation where the product of expansive Markov maps is conjugate to some subshift of finite type. We saw that several results from~\cite{BY} breakdown in this case, such as the dependence of escape rate on the minimal period of the hole. Several interesting questions can be asked in this framework. In general, does $a(q)$ determine the escape rate, as observed in Example~\ref{exam_123} for the maps $T_2$ and $T_3$. Are there any other (dynamical) factors that influence the escape rate into the hole other than size (length and number of the forbidden words) and position of the hole (correlations between forbidden words)? How to obtain the escape rate into an arbitrary hole in the torus which can be written as a limit of union of basic rectangles of the type $R_{i,j,m,n}$ (as the rectangles of this type form a sufficient semi-ring)? Can one characterize the holes in the torus such that the Hausdorff dimension of the survival set $\mathcal{W}=\{x\in I^k\ \vert\ T^nx\notin H,\ \text{for all}\ n \geq 0\}$ is non-zero? If $p_0$ is the initial probability measure on the state space $X$ and $\mathcal{W}_n$ be the complement of $\Omega_n$ for the hole $H\subseteq X$, then at what rate does $p_0(\mathcal{W}_n)$ decay? (note that $\mathcal{W}_{n+1}\subseteq \mathcal{W}_n$). Can we obtain some general results to compare the escape rate into holes for the maps that are conjugate to a subshift of finite type? 

A combinatorial question, independent of the underlying dynamics is: what is the largest number of words of length $m$ (fixed) with property (P) that can be constructed? Also, we observed that for $m>q$, the collection in Construction 1 is larger than the collection in Construction 2. Another question is whether Construction 1 is optimal. That is, does there exist a collection of $(q-1)^{m-1}+1$ words of length $m(>q)$ with property (P). Of course, as observed, the collection in Construction 1 cannot be expanded.

\section{Funding}
The research of the first author is supported by the Council of Scientific \& Industrial Research (CSIR), India (File no.~09/1020(0133)/2018-EMR-I), and the second author is supported by Center for Research on Environment and Sustainable Technologies (CREST), IISER Bhopal, CoE funded by the Ministry of Human Resource Development (MHRD), India.

\end{document}